\documentclass[11pt, a4paper]{amsart}
\usepackage[utf8]{inputenc}

\usepackage{dsfont}
\usepackage{polynom}
\usepackage{empheq}
\usepackage{subfiles}
\usepackage{systeme}
\usepackage{shuffle}
\usepackage{ytableau}
\usepackage{CJK}
\usepackage{bm}
\usepackage{faktor}
\usepackage{ifthen}
\usepackage{mathrsfs}
\usepackage{amsmath}
\usepackage{mathtools}
\usepackage{amssymb}
\usepackage{amscd}
\usepackage{appendix}
\usepackage{amsthm}
\usepackage{amsfonts}
\usepackage[all]{xy}
\usepackage{array}
\usepackage{enumerate}
\usepackage{graphicx}
\usepackage{longtable}
\usepackage[OT2,T1]{fontenc}
\usepackage{verbatim}
\usepackage{tikz}
\usetikzlibrary{matrix}
\usepackage{url,setspace,multirow,tikz-cd}
\DeclareSymbolFont{cyrletters}{OT2}{wncyr}{m}{n}
\DeclareMathSymbol{\Sha}{\mathalpha}{cyrletters}{"58}
\usepackage{xcolor}
\usepackage{textcomp}
\usepackage{manyfoot}
\usepackage{booktabs}
\usepackage{algorithm}
\usepackage{algorithmicx}
\usepackage{algpseudocode}
\usepackage{listings}
\usepackage{mathdots}
\usepackage{stmaryrd}
\usetikzlibrary{calc}
\usepackage{hyperref}

\theoremstyle{thmstyleone}
\newtheorem{thm}{Theorem}[section]

\newtheorem{lem}[thm]{Lemma}
\newtheorem{prop}[thm]{Proposition}%
\newtheorem{cor}[thm]{Corollary}

\theoremstyle{definition}
\newtheorem{eg}[thm]{Example}
\newtheorem{rem}[thm]{Remark}

\newtheorem{defn}[thm]{Definition}

\numberwithin{equation}{section}

\newcommand{\ZZ}{{\mathbb Z}}
\newcommand{\QQ}{{\mathbb Q}}

\newcommand{\CC}{{\mathbb C}}

\newcommand{\NN}{{\mathbb N}}

\newcommand{\PP}{{\mathbb P}}

\newcommand{\Lcal}{\mathcal{L}}
\newcommand{\Mcal}{\mathcal{M}}

\newcommand{\Tcal}{\mathcal{T}}

\newcommand{\Ical}{\mathcal{I}}

\newcommand{\Li}{\operatorname{Li}}

\begin{document}

\title[$p$-adic multiple zeta values of integer indices]{$p$-adic multiple zeta values of integer indices}

\author{Ku-Yu Fan}

\address{Graduate School of Mathematics, Nagoya University, Furo-cho, Chikusa-ku, Nagoya, 464-8602, Japan.}
\email{ku-yu.fan.d2@math.nagoya-u.ac.jp }

\date{March 24, 2026.}


\begin{abstract}
This paper concerns the $p$-adic multiple zeta values of integer indices that may contain zero or negative components. We introduce the admissibility and regularizability conditions for integer indices. We define the $p$-adic multiple zeta values associated with admissible integer indices to be finite rational linear combinations of $p$-adic multiple zeta values associated with admissible positive integer indices. We prove that the double shuffle relations, that is, the shuffle and stuffle product formulas, both hold for the values.
\end{abstract}



\maketitle

\tableofcontents

\section{Introduction}
Multiple zeta values are real numbers defined by the series
\begin{equation}\label{eq MZV}
  \zeta(k_1, \ldots, k_r) = \sum_{0<n_1<\cdots<n_r} \frac{1}{n_1^{k_1} \cdots n_r^{k_r}},
\end{equation}
where $(k_1,\ldots,k_{r})\in \ZZ_{>0}^r$ and $k_r > 1$ to ensure convergence. Multiple zeta values appear in various fields of mathematics such as analytic number theory \cite{Matsumoto}, algebraic geometry \cite{Goncharov}, mathematical physics \cite{Broadhurst}.

The study of multiple zeta values dates back to Euler, who proved the following formula \cite{Euler's formula}, for integers $a, b \geq 2$,
\begin{equation}\label{eq shuffle}
  \zeta(a) \zeta(b)=\sum_{i=0}^{a-1}\binom{i+b-1}{b-1} \zeta(a-i, b+i)+\sum_{i=0}^{b-1}\binom{i+a-1}{a-1} \zeta(b-i, a+i).
\end{equation}
This formula is now known as the shuffle product formula. On the other hand, for integers $a, b \geq 2$, the product of two series shows that
\begin{equation}\label{eq stuffle}
  \zeta(a) \zeta(b)=\zeta(a,b)+\zeta(b,a)+\zeta(a+b).
\end{equation}
This formula is called the stuffle (or quasi-shuffle/harmonic) product formula. The right-hand side of \eqref{eq shuffle} is equal to the right-hand side of \eqref{eq stuffle}, which gives us the double shuffle relation.

Multiple zeta values admit many analogues (e.g., $q$-analogue \cite{q-MZVDSR}, elliptic analogue \cite{EMZVDSR}). The $p$-adic multiple zeta values are one of the examples. The $p$-adic multiple zeta values are $p$-adic numbers defined by a certain limit
$$\zeta_p(k_1,\ldots,k_r) := \lim_{z\to 1}\Li^p_{k_1,\ldots,k_r}(z),$$
where $\Li^{p}_{k_1, \ldots, k_r}(z)$ denotes the $p$-adic multiple polylogarithms which are defined by
\begin{equation}\label{eq p-adic multiple polylogarithm}
  \Li^{p}_{k_1, \ldots, k_r}(z)=\sum_{0<n_1<\cdots<n_r} \frac{z^{n_r}}{n_1^{k_1} \cdots n_r^{k_r}}
\end{equation}
on $\{z\in \CC_p \big| |z|_p<1\}$ and, after fixing a branch of the $p$-adic logarithm, extend uniquely via Coleman’s $p$-adic iterated integrals to $\PP^1(\CC_p)\setminus\{0,1,\infty\}$. Besser and Furusho \cite{p-adicDS} proved the double shuffle relations for $p$-adic multiple zeta values for admissible indices (i.e., $k_r > 1$). Furusho and Jafari \cite{p-adicRDS} introduced a regularization for $p$-adic multiple zeta values with non-admissible indices (i.e.,  $k_r = 1$) and also proved the generalized double shuffle relations for the regularized values.

In this paper, we focus on integer indices $k = (k_1, \ldots, k_r) \in \ZZ^r$ and introduce both admissibility and regularizability conditions for integer indices in \S \ref{section Integer indices}. We show that multiple polylogarithms (see Definition \ref{defn MPL}) with admissible integer indices can be written as a $\QQ$-linear combination of multiple polylogarithms with admissible positive integer indices (see Proposition \ref{prop wt-dep>0}), and we use this result to define the positive-index map $\pi^+$ from the vector space spanned by admissible integer indices to the vector space spanned by admissible positive integer indices (see Definition \ref{defn pi+}) and show that it is compatible with the shuffle product (see Theorem \ref{thm shuffle-pi+}) and the stuffle product (see Theorem \ref{thm stuffle-pi+}). We define $p$-adic multiple zeta values with admissible integer indices to be their limits of the corresponding $p$-adic multiple polylogarithms \eqref{eq p-adic multiple polylogarithm} (see Definition  \ref{defn p-adic MZV integer index}). From Theorems \ref{thm shuffle-pi+} and \ref{thm stuffle-pi+} we deduce that $p$-adic multiple zeta values with admissible integer indices satisfy the double shuffle relations (see Theorem \ref{thm p-adic MZV integer index DSR}).

\section{Integer indices}\label{section Integer indices}
In this section, we introduce our integer index conventions and define the admissibility and regularizability conditions for integer indices.

\begin{defn}
  For $r\in \ZZ_{\geq 1}$, an $r$-tuple $(k_1, \ldots, k_r)$ in $\ZZ^r$ is called an {\it integer index} and denoted by $\bm{k}$. For $r =0$, the $0$-tuple is an {\it integer index} defined as $\emptyset$.
\end{defn}

\begin{defn}
  The set of (resp. positive) integer indices $\Ical$ (resp. $\Ical_{>0}$) is defined by
  $$\Ical = \bigsqcup_{r\in \ZZ_{\geq 0}} \ZZ^r \ (\text{resp. } \Ical_{>0} = \bigsqcup_{r\in \ZZ_{\geq 0}} \ZZ_{>0}^r).$$
\end{defn}

\begin{defn}
  For two integer indices $\bm{k} = (k_1, \ldots, k_r) \in \ZZ^r$ and $\bm{k}' = (k_1', \ldots, k_{r'}') \in \ZZ^{r'}$, we define the {\it concatenation} of $\bm{k}, \bm{k}'$ to be $(k_1, \ldots, k_r, k_1', \ldots, k_{r'}') \in \ZZ^{r+r'}$ and denote it by $(\bm{k}, \bm{k}')$.
\end{defn}

We define weight and depth of integer indices exactly as in the case of positive integers.

\begin{defn}
  For an integer index $\bm{k} = (k_1, \ldots, k_r) \in \ZZ^r$, we define the {\it weight} of $\bm{k}$ to be $k_1 + \cdots + k_r$ and denote it by $\mathrm{wt}(\bm{k})$.
\end{defn}

\begin{defn}
  For an integer index $\bm{k} = (k_1, \ldots, k_r) \in \ZZ^r$, we define the {\it depth} of $\bm{k}$ to be $r$ and denote it by $\mathrm{dep}(\bm{k})$.
\end{defn}

For an integer index $\bm{k} = (k_1, \ldots, k_r) \in \ZZ^r$, we introduce below the tail indices $\bm{k}_t \in \ZZ^{r-t+1}$ for $t = 1, \ldots, r$ and the index $m_{\bm{k}}$. We will see that $m_{\bm{k}}$ plays an important role in extending the admissibility condition to integer indices.

\begin{defn}
  For an integer index $\bm{k} = (k_1, \ldots, k_r) \in \ZZ^r$, we define the {\it tail index} $\bm{k}_t$ of $\bm{k}$ by $\bm{k}_t \coloneqq (k_t, \ldots, k_r) \in \ZZ^{r-t+1}$ for $t = 1, \ldots, r$.
\end{defn}

\begin{defn}\label{defn m}
  For an integer index $\bm{k} = (k_1, \ldots, k_r) \in \ZZ^r$, we define the {\it regularizability index} $m_{\bm{k}}$ of $\bm{k}$ by
  $$m_{\bm{k}} \coloneqq \min\left\{\mathrm{wt}(\bm{k}_t)-\mathrm{dep}(\bm{k}_t) \left| t = 1, \ldots, r\right.\right\} = \min\left\{ \sum_{i = t}^{r} (k_i - 1) \left| t = 1, \ldots, r\right.\right\}.$$
  For $\bm{k} = \emptyset$, we define $m_{\emptyset} \coloneqq \infty$.
\end{defn}

\begin{defn}
  Let $\Sigma$ be an element in the formal $\QQ$-linear space $\mathrm{span}_{\QQ}\{\Ical\}$. We define the support $\mathrm{Supp}(\Sigma)$ of $\Sigma$ to be the set of integer indices satisfying
  $$\Sigma = \sum_{\bm{k} \in \mathrm{Supp}(\Sigma)} c_{\bm{k}}\bm{k}$$
  with $c_{\bm{k}} \in \QQ \setminus \{0\}$.
\end{defn}

\begin{defn}
  Let $\Sigma$ be an element in the formal $\QQ$-linear space $\mathrm{span}_{\QQ}\{\Ical\}$. We define the {\it regularizability index} $m_{\Sigma}$ of $\Sigma$ by
  $$m_{\Sigma} \coloneqq \min\left\{ m_{\bm{k}} \left| \bm{k} \in \mathrm{Supp}(\Sigma) \right. \right\}.$$
\end{defn}

Admissibility is a necessary and sufficient condition for the convergence of multiple zeta values in the complex setting. For a given integer index, we introduce the following definitions of admissibility and regularizability.

\begin{defn}
  An integer index $\bm{k} = (k_1, \ldots, k_r) \in \ZZ^r$ is called {\it admissible} if $m_{\bm{k}}>0$.
\end{defn}

\begin{defn}
  The set of admissible (resp. positive) integer indices $\Ical^{\mathrm{adm}}$ (resp. $\Ical_{>0}^{\mathrm{adm}}$) is defined by
  $$\Ical^{\mathrm{adm}} = \{ \bm{k} \in \Ical | m_{\bm{k}}>0 \} \ (\text{resp. } \Ical_{>0}^{\mathrm{adm}} = \{ \bm{k} \in \Ical_{>0} | m_{\bm{k}}>0 \}).$$
\end{defn}

\begin{defn}
  An integer index $\bm{k} = (k_1, \ldots, k_r) \in \ZZ^r$ is called {\it regularizable} if $m_{\bm{k}} \geq 0$.
\end{defn}

\begin{defn}
  The set of regularizable (resp. positive) integer indices $\Ical^{\mathrm{reg}}$ (resp. $\Ical_{>0}^{\mathrm{reg}}$) is defined by
  $$\Ical^{\mathrm{reg}} = \{ \bm{k} \in \Ical | m_{\bm{k}} \geq 0 \} \ (\text{resp. } \Ical_{>0}^{\mathrm{reg}} = \{ \bm{k} \in \Ical_{>0} | m_{\bm{k}} \geq 0 \}).$$
\end{defn}

When $\bm{k} \in \NN^r$ is a positive integer index, $\bm{k}$ is admissible if $k_r>1$.

\begin{eg}
  \begin{enumerate}
    \item For $\bm{k} = (a) \in \ZZ$, we have $m_{\bm{k}} = \min\{a - 1\}$. So $\bm{k}$ is admissible if $a - 1>0$ and is regularizable if $a - 1 \geq 0$.
    \item For $\bm{k}' = (a, b) \in \ZZ^2$, we have $m_{\bm{k}'} = \min\{a+b-2, b-1\}$. Thus $\bm{k}'$ is admissible if $a+b-2, b - 1>0$ and is regularizable if $a+b-2, b - 1 \geq 0$.
    \item For $\bm{k}'' = (a, b, c) \in \ZZ^3$, we have $m_{\bm{k}''} = \min\{a+b+c-3, b+c-2, c-1\}$. Thus $\bm{k}'$ is admissible if $a+b+c-3, b+c-2, c-1>0$ and is regularizable if $a+b+c-3, b+c-2, c-1 \geq 0$.
  \end{enumerate}
\end{eg}

\section{MPLs of integer indices}\label{section MPLs of integer indices}
In this section, we recall the multiple polylogarithms and prove a theorem asserting that a multiple polylogarithm with an admissible integer index can be expressed as a linear combination of multiple polylogarithms with positive integer indices, and further every positive integer index appearing in the expansion is admissible. We use this result to define the positive-index map on the vector space spanned by indices.

\begin{defn}\label{defn MPL}
  Let $\bm{k} = (k_1, \ldots, k_r) \in \ZZ^r$ be an integer index. The {\it (single variable) multiple polylogarithms} are the power series defined by
  $$\Li_{k_1, \ldots, k_r}(z) \coloneqq \sum_{0<n_1<\cdots<n_r} \frac{z^{n_r}}{n_1^{k_1} \cdots n_r^{k_r}},$$
  where $z$ is a complex number.
\end{defn}

\begin{rem}
  Multiple polylogarithms converge for $\bigl\{z\in \CC \big| |z|<1 \bigr\}$, and the following equation holds.
  \begin{align*}
    & \frac{d}{dz}\Li_{k_1, \ldots, k_r}(z) = \begin{cases}
                                                \frac{1}{z}\Li_{k_1, \ldots, k_r-1}(z), & \mbox{if } k_r \neq 1, \\
                                                \frac{1}{1-z}\Li_{k_1, \ldots, k_{r-1}}(z), & \mbox{if } k_r = 1,
                                              \end{cases} \\
    & \frac{d}{dz}\Li_{1}(z) = \frac{1}{1-z}.
  \end{align*}
\end{rem}

We define the vector spaces generated by multiple polylogarithms.

\begin{defn}
  We consider the following $\QQ$-linear spaces of multiple polylogarithms:
  \begin{enumerate}[(i)]
    \item $\mathcal{MPL}_{>0}^{\mathrm{adm}} = \mathrm{span}_{\QQ}\left\{\Li_{\bm{k}}(z)\left|\bm{k} \in \Ical_{>0}^{\mathrm{adm}} \right.\right\} \subset \QQ\llbracket z \rrbracket.$
    \item $\mathcal{MPL}_{>0} = \mathrm{span}_{\QQ}\left\{\Li_{\bm{k}}(z)\left|\bm{k} \in \Ical_{>0} ( = \Ical_{>0}^{\mathrm{reg}}) \right.\right\} \subset \QQ\llbracket z \rrbracket.$
    \item $\mathcal{MPL}^{\mathrm{adm}} = \mathrm{span}_{\QQ}\left\{\Li_{\bm{k}}(z)\left|\bm{k} \in \Ical^{\mathrm{adm}} \right.\right\} \subset \QQ\llbracket z \rrbracket.$
    \item $\mathcal{MPL}^{\mathrm{reg}} = \mathrm{span}_{\QQ}\left\{\Li_{\bm{k}}(z)\left|\bm{k} \in \Ical^{\mathrm{reg}} \right.\right\} \subset \QQ\llbracket z \rrbracket.$
  \end{enumerate}
\end{defn}

We denote by $B^+_i$ (resp. $B^-_i$) be the Bernoulli numbers given by
$$\frac{t}{1 - e^{-t}} = \sum_{n = 0}^{\infty} B_{n}^{+} \frac{t^n}{n!} \ \left(\text{resp. } \frac{t}{e^t - 1} = \sum_{n = 0}^{\infty} B_{n}^{-} \frac{t^n}{n!}\right).$$
We note
$$B_1^+ = +\frac{1}{2}, \ B_1^- = -\frac{1}{2}, \ \text{and} \ B_n^- =  B_n^+ \ \text{for} \ n > 1.$$

\begin{lem}[Faulhaber’s formula]\label{lem Faulhaber's formula}
  Let $k$ be a non-negative integer. 
  $$\sum_{n=1}^{m} n^k=\frac{1}{k+1} \sum_{i=0}^k\binom{k+1}{i} B^+_i m^{k+1-i} \ \left(\text{resp. } \sum_{n=1}^{m-1} n^k=\frac{1}{k+1} \sum_{i=0}^k\binom{k+1}{i} B^-_i m^{k+1-i} - \delta_{k, 0} \right).$$
\end{lem}

The preceding lemma is the key to the following proposition.

\begin{prop}\label{prop wt-dep>0}
  The following equalities hold:
  \begin{enumerate}[(i)]
    \item $\mathcal{MPL}_{>0}^{\mathrm{adm}} = \mathcal{MPL}^{\mathrm{adm}}.$
    \item $\mathcal{MPL}_{>0} = \mathcal{MPL}^{\mathrm{reg}}.$
  \end{enumerate}  
\end{prop}

\begin{proof}
  First, we have the clear inclusions:
  \begin{align*}
    \mathcal{MPL}_{>0}^{\mathrm{adm}} \subset & \mathcal{MPL}^{\mathrm{adm}}, \\
    \mathcal{MPL}_{>0} \subset & \mathcal{MPL}^{\mathrm{reg}}.
  \end{align*}
  Next, we will below prove the opposite inclusion. Let $\bm{k} = (k_1, \ldots, k_r) \in \ZZ^r$ be an admissible (resp. regularizable) integer index. We prove the multiple polylogarithm $\Li_{\bm{k}}(z) \in \mathcal{MPL}^{\mathrm{adm}}$ (resp. $\mathcal{MPL}^{\mathrm{reg}}$) belongs to $\mathcal{MPL}_{>0}^{\mathrm{adm}}$ (resp. $\mathcal{MPL}_{>0}$) by induction on the depth $r$ of $\bm{k}$. For $r=1$, admissible (resp. regularizable) implies $k_1 \in \NN_{>1}$ (resp. $k_1 \in \NN$); thus $\Li_{k_1}(z)$ lies in $\mathcal{MPL}_{>0}^{\mathrm{adm}}$ (resp. $\mathcal{MPL}_{>0}$).\\
   
  Assume the statement holds for all admissible (resp. regularizable) integer indices $\bm{k}$ with $\mathrm{dep}(\bm{k})<r$. We prove it for all admissible (resp. regularizable) integer indices $\bm{k}$ with $\mathrm{dep}(\bm{k}) = r$. If $\bm{k}$ is a positive integer index, then we are done. We let $m$ be the minimal element of the set $\{i \in \{1, \ldots, r\} \,|\, k_i \leq 0\}$ if $\bm{k}$ is not a positive integer index, and set $n_{m - 1} = 0$ when $m = 1$. We use Lemma \ref{lem Faulhaber's formula} to compute the multiple polylogarithm as follows:
  \begin{align*}
    & \Li_{\bm{k}}(z) = \sum_{0<n_1<\cdots<n_r} \frac{z^{n_r}}{n_1^{k_1} \cdots n_r^{k_r}} = \sum_{0<n_1<\cdots<n_{m}<\cdots<n_r} \frac{z^{n_r}}{n_1^{k_1} \cdots n_{m}^{k_{m}}  \cdots n_r^{k_r}} \\
    = & \sum_{\substack{0<n_1<\cdots<n_{m-1}<n_{m-1}+1 \\<n_{m+1}<\cdots<n_r}} \left( \sum_{n = n_{m-1}+1}^{n_{m+1}-1} n^{-k_{m}} \right) \frac{z^{n_r}}{n_1^{k_1} \cdots n_{m-1}^{k_{m-1}} \cdot n_{m+1}^{k_{m+1}}  \cdots n_r^{k_r}} \\
    = & \sum_{\substack{0<n_1<\cdots<n_{m-1}\\<n_{m+1}<\cdots<n_r}} \left( \sum_{n = 1}^{n_{m+1}-1} n^{-k_{m}} - \sum_{n = 1}^{n_{m-1}} n^{-k_{m}} \right) \frac{z^{n_r}}{n_1^{k_1} \cdots n_{m-1}^{k_{m-1}} \cdot n_{m+1}^{k_{m+1}}  \cdots n_r^{k_r}} \\
    = & \sum_{\substack{0<n_1<\cdots<n_{m-1}\\<n_{m+1}<\cdots<n_r}} \frac{1}{-k_{m}+1} \sum_{i=0}^{-k_{m}}\binom{-k_{m}+1}{i} B^-_i n_{m+1}^{-k_{m}+1-i} \frac{z^{n_r}}{n_1^{k_1} \cdots n_{m-1}^{k_{m-1}} \cdot n_{m+1}^{k_{m+1}}  \cdots n_r^{k_r}} \\
    & - \delta_{k_m, 0} \frac{z^{n_r}}{n_1^{k_1} \cdots n_{m-1}^{k_{m-1}} \cdot n_{m+1}^{k_{m+1}}  \cdots n_r^{k_r}} \\
    & - \sum_{\substack{0<n_1<\cdots<n_{m-1}\\<n_{m+1}<\cdots<n_r}} \frac{1}{-k_{m}+1} \sum_{i=0}^{-k_{m}}\binom{-k_{m}+1}{i} B^+_i n_{m-1}^{-k_{m}+1-i} \frac{z^{n_r}}{n_1^{k_1} \cdots n_{m-1}^{k_{m-1}} \cdot n_{m+1}^{k_{m+1}}  \cdots n_r^{k_r}} \\
   = & \frac{1}{-k_{m}+1} \sum_{i=0}^{-k_{m}}\binom{-k_{m}+1}{i} B^-_i \sum_{\substack{0<n_1<\cdots<n_{m-1}\\<n_{m+1}<\cdots<n_r}} \frac{z^{n_r}}{n_1^{k_1} \cdots n_{m-1}^{k_{m-1}} \cdot n_{m+1}^{k_{m+1}+k_{m}-1+i}  \cdots n_r^{k_r}} \\
    & - \delta_{k_m, 0} \sum_{\substack{0<n_1<\cdots<n_{m-1}\\<n_{m+1}<\cdots<n_r}} \frac{z^{n_r}}{n_1^{k_1} \cdots n_{m-1}^{k_{m-1}} \cdot n_{m+1}^{k_{m+1}}  \cdots n_r^{k_r}} \\
    & - \frac{1}{-k_{m}+1} \sum_{i=0}^{-k_{m}}\binom{-k_{m}+1}{i} B^+_i \sum_{\substack{0<n_1<\cdots<n_{m-1}\\<n_{m+1}<\cdots<n_r}} \frac{z^{n_r}}{n_1^{k_1} \cdots n_{m-1}^{k_{m-1}+k_{m}-1+i} \cdot n_{m+1}^{k_{m+1}}  \cdots n_r^{k_r}} \\
    = & \frac{1}{-k_{m}+1} \sum_{i=0}^{-k_{m}}\binom{-k_{m}+1}{i} B^-_i \Li_{k_1, \ldots, k_{m-1}, k_{m+1}+k_{m}-1+i, k_{m+2}, \ldots, k_r}(z) \\
    & - \delta_{k_m, 0} \Li_{k_1, \ldots, k_{m-1}, k_{m+1}, \ldots, k_r}(z) \\
    & - \frac{1}{-k_{m}+1} \sum_{i=0}^{-k_{m}}\binom{-k_{m}+1}{i} B^+_i \Li_{k_1, \ldots, k_{m-2}, k_{m-1}+k_{m}-1+i, k_{m+1}, \ldots, k_r}(z)
  \end{align*}
  We use the following notation
  \begin{align*}
    \bm{k}_{\widehat{m},i}^+ = & (k_1, \ldots, k_{m-1}, k_{m+1}+k_{m}-1+i, k_{m+2}, \ldots, k_r) \\
    \bm{k}_{\widehat{m},i}^- = & (k_1, \ldots, k_{m-2}, k_{m-1}+k_{m}-1+i, k_{m+1}, \ldots, k_r) \\
    \bm{k}_{\widehat{m}} = & (k_1, \ldots, k_{m-1}, k_{m+1}, \ldots, k_r).
  \end{align*}
  Then, it is easy to check
  \begin{align*}
    & \mathrm{wt}((\bm{k}_{\widehat{m},i}^+)_{t})-\mathrm{dep}((\bm{k}_{\widehat{m},i}^+)_{t}) = \mathrm{wt}(\bm{k}_{t+1})-\mathrm{dep}(\bm{k}_{t+1}), & \text{ if } & t = m+1, \ldots, r-1, \\
    & \mathrm{wt}((\bm{k}_{\widehat{m},i}^+)_{t})-\mathrm{dep}((\bm{k}_{\widehat{m},i}^+)_{t}) = \mathrm{wt}(\bm{k}_{t})-\mathrm{dep}(\bm{k}_{t}) + i, & \text{ if } & t = m, \\
    & \mathrm{wt}((\bm{k}_{\widehat{m},i}^+)_{t})-\mathrm{dep}((\bm{k}_{\widehat{m},i}^+)_{t}) \geq \mathrm{wt}((\bm{k}_{\widehat{m},i}^+)_{t+1})-\mathrm{dep}((\bm{k}_{\widehat{m},i}^+)_{t+1}), & \text{ if } & t = 1, \ldots, m-1, \\
    & \mathrm{wt}((\bm{k}_{\widehat{m},i}^-)_{t})-\mathrm{dep}((\bm{k}_{\widehat{m},i}^-)_{t}) = \mathrm{wt}(\bm{k}_{t+1})-\mathrm{dep}(\bm{k}_{t+1}), & \text{ if } & t = m, \ldots, r-1, \\
    & \mathrm{wt}((\bm{k}_{\widehat{m},i}^-)_{t})-\mathrm{dep}((\bm{k}_{\widehat{m},i}^-)_{t}) = \mathrm{wt}(\bm{k}_{t+1})-\mathrm{dep}(\bm{k}_{t+1}) + k_{m-1} -1 +i, & \text{ if } & t = m-1, \\
    & \mathrm{wt}((\bm{k}_{\widehat{m},i}^-)_{t})-\mathrm{dep}((\bm{k}_{\widehat{m},i}^-)_{t}) \geq \mathrm{wt}((\bm{k}_{\widehat{m},i}^-)_{t+1})-\mathrm{dep}((\bm{k}_{\widehat{m},i}^-)_{t+1}), & \text{ if } & t = 1, \ldots, m-2\\
    & \mathrm{wt}((\bm{k}_{\widehat{m}})_{t})-\mathrm{dep}((\bm{k}_{\widehat{m}})_{t}) = \mathrm{wt}(\bm{k}_{t+1})-\mathrm{dep}(\bm{k}_{t+1}), & \text{ if } & t = m, \ldots, r-1, \\
    & \mathrm{wt}((\bm{k}_{\widehat{m}})_{t})-\mathrm{dep}((\bm{k}_{\widehat{m}})_{t}) \geq \mathrm{wt}(\bm{k}_{t})-\mathrm{dep}(\bm{k}_{t}), & \text{ if } & t = 1, \ldots, m-1, \\
  \end{align*}
  Hence, we obtain that $\bm{k}_{\widehat{m},i}^+$, $\bm{k}_{\widehat{m},i}^-$, and $\bm{k}_{\widehat{m}}$ are admissible (resp. regularizable) because we assume that $\bm{k}$ is admissible (resp. regularizable); hence $\Li_{\bm{k}_{\widehat{m},i}^+}(z), \Li_{\bm{k}_{\widehat{m},i}^-}(z) \in \mathcal{MPL}_{>0}^{\mathrm{adm}}$ (resp. $\mathcal{MPL}_{>0}$). This implies $\Li_{\bm{k}}(z) \in \mathcal{MPL}_{>0}^{\mathrm{adm}}$ (resp. $\mathcal{MPL}_{>0}$) by our computation, which completes the proof of this proposition.
\end{proof}

\begin{defn}\label{defn pi+}
  The {\it positive-index map} $\pi^{+}$ is the $\QQ$-linear map
  $$\pi^{+} : \mathrm{span}_{\QQ}\{\Ical^{\mathrm{adm}}\} \to \mathrm{span}_{\QQ}\{\Ical_{>0}^{\mathrm{adm}}\} \ (\text{resp. } \mathrm{span}_{\QQ}\{\Ical^{\mathrm{reg}}\} \to \mathrm{span}_{\QQ}\{\Ical_{>0}^{\mathrm{reg}}\})$$
  defined by
  $$\pi^{+}(\bm{k}) = \sum_{\bm{l}}c_{\bm{k},\bm{l}}\bm{l},$$
  where the positive integer indices $\bm{l}$ and the coefficients $c_{\bm{k},\bm{l}}\in \QQ \setminus \{0\}$ are uniquely obtained step by step through the inductive argument in the proof of Proposition \ref{prop wt-dep>0}.
\end{defn}

\begin{rem}\label{rem pi+}
  Actually, the positive-index map can be extended to $\mathrm{span}_{\QQ}\{\Ical\}$, that is,
  $$\pi^{+} : \mathrm{span}_{\QQ}\{\Ical\} \to \mathrm{span}_{\QQ}\{\Ical\} , \qquad \pi^{+}(\bm{k}) = \sum_{\bm{l}}c_{\bm{k},\bm{l}}\bm{l},$$
  where the integer indices
  $$\bm{l} \in \{\emptyset\} \sqcup \bigsqcup_{r\in \ZZ_{> 0}} \ZZ_{>0}^{r-1}\times \ZZ_{>m_{\bm{k}}}$$
  and the coefficients $c_{\bm{k},\bm{l}}\in \QQ \setminus \{0\}$ are obtained step by step through the same inductive argument as in the proof of Proposition \ref{prop wt-dep>0}. For integer indices $\bm{k} \in \Ical$, we have the identity
  \begin{equation}\label{eq pi+}
    \Li_{\bm{k}}(z) = \sum_{\bm{l} \in \mathrm{Supp}(\pi^{+}(\bm{k}))}c_{\bm{k},\bm{l}}\Li_{\bm{l}}(z) \in \QQ\llbracket z \rrbracket.
  \end{equation}
\end{rem}

\begin{eg}
  \begin{enumerate}
    \item For $\bm{k} = (a) \in \ZZ$, we have $\pi^{+}(\bm{k}) = (a)$.
    \item For $\bm{k} = (a, b) \in \ZZ^2$, we have $\pi^{+}(\bm{k}) = (a, b)$ if $a>0$ and
    $$\pi^{+}(\bm{k}) = \frac{1}{-a+1} \sum_{i=0}^{-a}\binom{-a+1}{i} B^-_i (a+b-1+i) - \delta_{-a, 0}(b)$$
    if $a \leq 0$.
  \end{enumerate}
\end{eg}

\section{Shuffle product for integer indices}\label{section Shuffle product for integer indices}
In this section, we recall the shuffle product of multiple polylogarithms with integer indices and prove that the shuffle product of two admissible (resp. regularizable) integer indices can be expanded as a linear combination of integer indices that remain admissible (resp. regularizable). We show that the positive-index map is an algebra homomorphism with respect to shuffle product (Theorem \ref{thm shuffle-pi+}).

We start by recalling the shuffle product for integer indices introduced by Ebrahimi-Fard, Manchon, and Singer \cite{non-positive shuffle}.
\begin{defn}[{\cite{non-positive shuffle}}]
  Let $X$ be the alphabet $\{j, d, y\}$, and let $W$ denote the set of words on the alphabet $X$, subject to the rule $jd = dj = \bm{1}$, where $\bm{1}$ denotes the empty word. We define $A$ as the $\QQ$-linear space $\mathrm{span}_{\QQ} \{ W \}$ spanned by the words in $W$.
\end{defn}

Every $w\in W$ can be written uniquely in the canonical form
$$w = j^{k_1}yj^{k_2}y\cdots j^{k_{r-1}}yj^{k_r}$$
with $k_i \in \ZZ$; here we adopt the conventions $j^{-1} = d$ and $j^0 = \bm{1}$.

\begin{defn}[{\cite{non-positive shuffle}}]
  Let $w = j^{k_1}yj^{k_2}y\cdots j^{k_{r-1}}yj^{k_r}$ be a word in $W$. We define the {\it length} of $w$ by $|w| = |k_1| + \cdots + |k_r| + (r - 1)$.
\end{defn}

\begin{defn}[{\cite{non-positive shuffle}}]\label{defn non-positive shuffle}
  We define the shuffle product $\shuffle : A \otimes A \to A$ by $\bm{1} \shuffle w = w \shuffle \bm{1} = w$ for any word $w \in W$, and recursively with respect to the sum of the lengths of two words in $W$:
  \begin{enumerate}[(i)]
    \item $yu \shuffle v \coloneqq u \shuffle yv \coloneqq y(u \shuffle v),$
    \item $ju \shuffle jv \coloneqq  j(u \shuffle jv) + j(ju \shuffle v),$
    \item $du \shuffle dv \coloneqq d(u \shuffle dv) - u \shuffle d^2v,$
    \item $du \shuffle jv \coloneqq d(u \shuffle jv) - u \shuffle v,$
    \item $ju \shuffle dv \coloneqq d(ju \shuffle v) - u \shuffle v.$
  \end{enumerate}
\end{defn}

\begin{lem}[{\cite[Lemma~3.4]{non-positive shuffle}}]
  The $\QQ$-vector space
  $$\Tcal \coloneqq \mathrm{span}_{\QQ} \left\{ j^{k_1}yj^{k_2}y\cdots j^{k_{r-1}}yj^{k_r} \in W \left| k_1, \ldots, k_{r} \in \ZZ, k_r \neq 0, r \in \NN \right. \right\}$$
  is a two-sided ideal of $(A, \shuffle)$.
\end{lem}

\begin{proof}
  See {\cite[Lemma~3.4]{non-positive shuffle}}.
\end{proof}

\begin{defn}
  Let $\Lcal$ be the smallest two-sided ideal of $A' \coloneqq A/\Tcal$ stable under left concatenation with $j$ or $d$ and containing the expressions
  $$d(u \shuffle v) - du \shuffle v - u \shuffle dv, \null u, v \in Wy.$$
  We define the quotient algebra $B$ as $A'/\Lcal \cong A/(\Tcal + \Lcal)$.
\end{defn}

\begin{prop}[{\cite[Proposition~3.5]{non-positive shuffle}}]\label{prop non-positive shuffle}
  The pair $(B, \shuffle)$ is a commutative, associative, and unital algebra.
\end{prop}

\begin{proof}
  See {\cite[Proposition~3.5]{non-positive shuffle}}.
\end{proof}

After passing to the quotient $A'$ (hence to $B$), every quotient class admits a unique representative in $Wy$. More precisely, for any quotient class $\overline{w}\in B$, there exists a word $w' \in \overline{w}$ that can be written uniquely in the canonical form
$$w' = j^{k_r}y\cdots j^{k_2}yj^{k_1}y$$
which corresponds to the integer index $(k_1, \ldots, k_r)$ in $\ZZ^r$. This gives a bijection $\phi : \mathrm{span}_{\QQ}\{\Ical\} \to B$ between the $\QQ$-linear space $\mathrm{span}_{\QQ}\{\Ical\}$ and $B$. Consequently, the shuffle product of integer indices follows from $(B, \shuffle)$, that is, $(\mathrm{span}_{\QQ}\{\Ical\}, \shuffle)$ is a $\QQ$-algebra with the shuffle product defined by
$$\phi^{-1} \circ \shuffle \circ (\phi, \phi) : \mathrm{span}_{\QQ}\{\Ical\} \times \mathrm{span}_{\QQ}\{\Ical\} \to \mathrm{span}_{\QQ}\{\Ical\}.$$

\begin{rem}\label{rem non-positive shuffle}
  It follows from Proposition \ref{prop non-positive shuffle} that the validity of three defining relations (iii)–(v) are reduced to the validity of
  $$du \shuffle v = d(u \shuffle v) - u \shuffle dv$$
  for all words $u, v \in B$.
\end{rem}

\begin{defn}
  Let $w = j^{k_r}y\cdots j^{k_2}yj^{k_1}y$ be a word in $B$. We define $\bm{k}_{w} = (k_1, \ldots, k_r)$ to be the {\it corresponding integer index} in $\ZZ^r$.
  Conversely for $\bm{k} = (k_1, \ldots, k_r) \in \ZZ^r$, we define $w_{\bm{k}} = j^{k_r}y\cdots j^{k_2}yj^{k_1}y$ to be the {\it corresponding word} in $B$.
\end{defn}

\begin{lem}\label{lem m_k shuffle}
  Let $w$ be a word in $B$. Then,
  \begin{enumerate}
    \item $m_{\bm{k}_{yw}} = \min \{-1, m_{\bm{k}_{w}}-1\}$,
    \item $m_{\bm{k}_{jw}} = m_{\bm{k}_{w}}+1$,
    \item $m_{\bm{k}_{dw}} = m_{\bm{k}_{w}}-1$.
  \end{enumerate}
\end{lem}

\begin{proof}
  We rewrite $\bm{k}_{w}$ as $(k_1, \ldots, k_r)$. By Definition \ref{defn m}, we obtain
  $$m_{\bm{k}_{yw}} = m_{(k_1, \ldots, k_r, 0)} = \min \{-1, m_{\bm{k}_{w}}-1\},$$
  $$m_{\bm{k}_{jw}} = m_{(k_1, \ldots, k_r + 1)} = m_{\bm{k}_{w}}+1,$$
  $$m_{\bm{k}_{dw}} = m_{(k_1, \ldots, k_r - 1)} = m_{\bm{k}_{w}}-1.$$
\end{proof}

We will use Lemma \ref{lem m_k shuffle} several times in the proof of Proposition \ref{prop main shuffle}.

\begin{prop}\label{prop main shuffle}
  Let $\bm{k}, \bm{k}'$ be two integer indices. Then, 
  $$m_{\bm{k} \shuffle \bm{k}'} = \min \{ m_{\bm{k}}, m_{\bm{k}'}, m_{\bm{k}}+m_{\bm{k}'}\}.$$
\end{prop}

\begin{proof}
  We prove the claim by induction on the sum of the lengths of $w_{\bm{k}}$ and $w_{\bm{k}'}$. For $|w_{\bm{k}}| + |w_{\bm{k}'}| = 0$, it is trivial. For $|w_{\bm{k}}| + |w_{\bm{k}'}| = 1$, it suffices to prove the case for $y \shuffle \bm{1}$, that is,
  $$m_{(0) \shuffle \emptyset} = \min \{ m_{(0)}, m_{\emptyset}, m_{(0)}+m_{\emptyset}\},$$
  and it follows that $m_{(0)} = -1$ and $m_{\emptyset} = \infty$. For $|w_{\bm{k}}| + |w_{\bm{k}'}| = 2$,  it suffices to prove the four cases for
  $$jy \shuffle \bm{1}, \ dy \shuffle \bm{1}, \ yy \shuffle \bm{1}, \ y \shuffle y.$$
  It is easy to check
  $$m_{(1) \shuffle \emptyset} = 0 = \min \{ m_{(1)}, m_{\emptyset}, m_{(1)}+m_{\emptyset}\},$$
  $$m_{(-1) \shuffle \emptyset} = -2 = \min \{ m_{(-1)}, m_{\emptyset}, m_{(-1)}+m_{\emptyset}\},$$
  $$m_{(0, 0) \shuffle \emptyset} = -2 = \min \{ m_{(0, 0)}, m_{\emptyset}, m_{(0, 0)}+m_{\emptyset}\},$$
  $$m_{(0) \shuffle (0)} = -2 = \min \{ m_{(0)}, m_{(0)}, m_{(0)}+m_{(0)}\}.$$
  
  Assume $m_{\bm{k} \shuffle \bm{k}'} = \min \{ m_{\bm{k}}, m_{\bm{k}'}, m_{\bm{k}}+m_{\bm{k}'}\}$ holds for two integer indices $\bm{k}, \bm{k}'$ which satisfy $|w_{\bm{k}}| + |w_{\bm{k}'}| < n$. We prove the statement when $|w_{\bm{k}}| + |w_{\bm{k}'}| = n$ using Definition \ref{defn non-positive shuffle} and Remark \ref{rem non-positive shuffle}.
  \begin{enumerate}[\underline{\text{Case}}]
    \item 1. (i) in Definition \ref{defn non-positive shuffle}, i.e., $w_{\bm{k}} = yu$ and $w_{\bm{k}'} = v$. In this case, we obtain
        $$w_{\bm{k}} \shuffle w_{\bm{k}'} = yu \shuffle v = y(u \shuffle v).$$
        By inductive hypothesis, $m_{\bm{k}_{u} \shuffle \bm{k}_{v}} = \min \{ m_{\bm{k}_{u}}, m_{\bm{k}_{v}}, m_{\bm{k}_{u}}+m_{\bm{k}_{v}}\}$.
        \begin{enumerate}[\underline{\text{Case}}]
          \item 1.1. $m_{\bm{k}_{u} \shuffle \bm{k}_{v}} \geq 0$. In this case, we obtain $m_{\bm{k}_{u}} \geq 0$ and $m_{\bm{k}_{v}} \geq 0$. This implies
              $$m_{\bm{k} \shuffle \bm{k}'} = \min\{-1, m_{\bm{k}_{u} \shuffle \bm{k}_{v}}-1\} = -1$$
              and
              $$\min \{m_{\bm{k}_{yu}}, m_{\bm{k}_{v}}, m_{\bm{k}_{yu}}+m_{\bm{k}_{v}}\} = \min \{-1, m_{\bm{k}_{v}}, -1+m_{\bm{k}_{v}}\} = -1.$$
          \item 1.2. $m_{\bm{k}_{u} \shuffle \bm{k}_{v}} < 0$. In this case, we obtain
                    $$m_{\bm{k} \shuffle \bm{k}'} = \min\{-1, m_{\bm{k}_{u} \shuffle \bm{k}_{v}}-1\} = m_{\bm{k}_{u} \shuffle \bm{k}_{v}}-1.$$
                    We prove $\min \{m_{\bm{k}_{yu}}, m_{\bm{k}_{v}}, m_{\bm{k}_{yu}}+m_{\bm{k}_{v}}\} = m_{\bm{k}_{u} \shuffle \bm{k}_{v}}-1$ by considering the following cases.
              \begin{enumerate}[\underline{\text{Case}}]
                \item 1.2.1. $m_{\bm{k}_{u}} \geq 0$ and $m_{\bm{k}_{v}} < 0$. In this case, we obtain
                    $$\min \{m_{\bm{k}_{yu}}, m_{\bm{k}_{v}}, m_{\bm{k}_{yu}}+m_{\bm{k}_{v}}\} = \min \{-1, m_{\bm{k}_{v}}, -1+m_{\bm{k}_{v}}\} = m_{\bm{k}_{v}}-1 = m_{\bm{k}_{u} \shuffle \bm{k}_{v}}-1.$$
                \item 1.2.2. $m_{\bm{k}_{u}} < 0$ and $m_{\bm{k}_{v}} \geq 0$. In this case, we obtain
                    $$\min \{m_{\bm{k}_{yu}}, m_{\bm{k}_{v}}, m_{\bm{k}_{yu}}+m_{\bm{k}_{v}}\} = \min \{m_{\bm{k}_{u}}-1, m_{\bm{k}_{v}}, m_{\bm{k}_{u}}-1+m_{\bm{k}_{v}}\} = m_{\bm{k}_{u}}-1 = m_{\bm{k}_{u} \shuffle \bm{k}_{v}}-1.$$
                \item 1.2.3. $m_{\bm{k}_{u}} < 0$ and $m_{\bm{k}_{v}} < 0$. In this case, we obtain
                    $$\min \{m_{\bm{k}_{yu}}, m_{\bm{k}_{v}}, m_{\bm{k}_{yu}}+m_{\bm{k}_{v}}\} = \min \{m_{\bm{k}_{u}}-1, m_{\bm{k}_{v}}, m_{\bm{k}_{u}}-1+m_{\bm{k}_{v}}\} = m_{\bm{k}_{u}}+m_{\bm{k}_{v}}-1 = m_{\bm{k}_{u} \shuffle \bm{k}_{v}}-1.$$
              \end{enumerate}
        \end{enumerate}
    \item 2. (ii) in Definition \ref{defn non-positive shuffle}, i.e., $w_{\bm{k}} = ju$ and $w_{\bm{k}'} = jv$. In this case, we obtain
        $$w_{\bm{k}} \shuffle w_{\bm{k}'} = ju \shuffle jv = j(u \shuffle jv) + j(ju \shuffle v).$$
        By inductive hypothesis,
        $$m_{\bm{k}_{ju} \shuffle \bm{k}_{v}} = \min \{ m_{\bm{k}_{ju}}, m_{\bm{k}_{v}}, m_{\bm{k}_{ju}}+m_{\bm{k}_{v}}\}$$
        and
        $$m_{\bm{k}_{u} \shuffle \bm{k}_{jv}} = \min \{ m_{\bm{k}_{u}}, m_{\bm{k}_{jv}}, m_{\bm{k}_{u}}+m_{\bm{k}_{jv}}\}.$$
        This implies
        \begin{align*}
          m_{\bm{k} \shuffle \bm{k}'} = & \min \{1+m_{\bm{k}_{ju} \shuffle \bm{k}_{v}}, 1+m_{\bm{k}_{u} \shuffle \bm{k}_{jv}}\} \\
          = & \min \{1+\min \{ m_{\bm{k}_{ju}}, m_{\bm{k}_{v}}, m_{\bm{k}_{ju}}+m_{\bm{k}_{v}}\}, 1+\min \{ m_{\bm{k}_{u}}, m_{\bm{k}_{jv}}, m_{\bm{k}_{u}}+m_{\bm{k}_{jv}}\}\} \\
          = & \min \{m_{\bm{k}_{ju}}+1, m_{\bm{k}_{v}}+1, m_{\bm{k}_{ju}}+m_{\bm{k}_{v}}+1, m_{\bm{k}_{u}}+1, m_{\bm{k}_{jv}}+1, m_{\bm{k}_{u}}+m_{\bm{k}_{jv}}+1\} \\
          = & \min \{m_{\bm{k}_{ju}}+1, m_{\bm{k}_{jv}}, m_{\bm{k}_{ju}}+m_{\bm{k}_{jv}}, m_{\bm{k}_{ju}}, m_{\bm{k}_{jv}}+1, m_{\bm{k}_{ju}}+m_{\bm{k}_{jv}}\} \\
          = & \min \{m_{\bm{k}_{ju}}, m_{\bm{k}_{jv}}, m_{\bm{k}_{ju}}+m_{\bm{k}_{jv}}\}.
        \end{align*}
    \item 3. (iii)–(v)  in Definition \ref{defn non-positive shuffle}, i.e., $w_{\bm{k}} = du$, where we used Remark \ref{rem non-positive shuffle}. We consider three cases as follows.
        \begin{enumerate}[\underline{\text{Case}}]
          \item 3.1. $w_{\bm{k}'} = yv$. In this case, since shuffle product is commutative, this is Case 1.
          \item 3.2. $w_{\bm{k}'} = jv$. In this case, we obtain
              $$du \shuffle jv = d(u \shuffle jv) - u \shuffle djv = d(u \shuffle jv) - u \shuffle v.$$
              By inductive hypothesis, we obtain
              \begin{align*}
                m_{\bm{k} \shuffle \bm{k}'} = & \min \{ m_{\bm{k}_{d(u \shuffle jv)}}, m_{\bm{k}_{u \shuffle v}}\} \\
                = & \min \{ \min \{ m_{\bm{k}_{u}}, m_{\bm{k}_{jv}}, m_{\bm{k}_{u}}+m_{\bm{k}_{jv}}\}-1, \min \{ m_{\bm{k}_{u}}, m_{\bm{k}_{v}}, m_{\bm{k}_{u}}+m_{\bm{k}_{v}}\}\} \\
                = & \min \{ m_{\bm{k}_{u}}-1, m_{\bm{k}_{jv}}-1, m_{\bm{k}_{u}}+m_{\bm{k}_{jv}}-1, m_{\bm{k}_{u}}, m_{\bm{k}_{v}}, m_{\bm{k}_{u}}+m_{\bm{k}_{v}}\} \\
                = & \min \{ m_{\bm{k}_{du}}, m_{\bm{k}_{jv}}-1, m_{\bm{k}_{du}}+m_{\bm{k}_{jv}}\} \\
                = & \min \{ m_{\bm{k}_{du}}, m_{\bm{k}_{du}}+m_{\bm{k}_{jv}}\} \\
                = & \min \{ m_{\bm{k}_{du}}, m_{\bm{k}_{jv}}, m_{\bm{k}_{du}}+m_{\bm{k}_{jv}}\}.
              \end{align*}
              In the penultimate, the equality follows from $m_{\bm{k}_{du}} \leq -1$, and in the last, the equality follows from $m_{\bm{k}_{du}} \leq 0$.
          \item 3.3. $w_{\bm{k}'} = dv$.  In this case, we rewrite $w_{\bm{k}} = d^n yu'$. Then,
              $$w_{\bm{k}} \shuffle w_{\bm{k}'} = d^n yu' \shuffle dv = \sum_{i=1}^{n} (-1)^{i-1}d(d^{n-i} yu' \shuffle d^{i}v) + (-1)^n yu' \shuffle d^{n+1}v.$$
              By inductive hypothesis,
              $$m_{\bm{k}_{d^{n-i} yu'} \shuffle \bm{k}_{d^{i}v}} = \min \{ m_{\bm{k}_{d^{n-i} yu'}}, m_{\bm{k}_{d^{i}v}}, m_{\bm{k}_{d^{n-i} yu'}}+m_{\bm{k}_{d^{i}v}}\}$$
              holds for $i = 1, \ldots, n$. By Case 1,
              $$m_{\bm{k}_{yu'} \shuffle \bm{k}_{d^{n+1}v}} = \min \{ m_{\bm{k}_{yu'}}, m_{\bm{k}_{d^{n+1}v}}, m_{\bm{k}_{yu'}}+m_{\bm{k}_{d^{n+1}v}}\}.$$
              These imply
              \begin{align*}
                m_{\bm{k} \shuffle \bm{k}'} = & \min \{\min \{m_{\bm{k}_{d^{n-i} yu'} \shuffle \bm{k}_{d^{i}v}} -1 | i = 1, \ldots, n\}, m_{\bm{k}_{yu'} \shuffle \bm{k}_{d^{n+1}v}}\} \\
                = & \min \{\min \{\min \{ m_{\bm{k}_{d^{n-i} yu'}}-1, m_{\bm{k}_{d^{i}v}}-1, m_{\bm{k}_{d^{n-i} yu'}}+m_{\bm{k}_{d^{i}v}}-1\} | i = 1, \ldots, n\}, \\
                & \min \{ m_{\bm{k}_{yu'}}, m_{\bm{k}_{d^{n+1}v}}, m_{\bm{k}_{yu'}}+m_{\bm{k}_{d^{n+1}v}}\}\} \\
                = & \min \{\min \{\min \{ m_{\bm{k}_{yu'}}-n+i-1, m_{\bm{k}_{v}}-i-1, m_{\bm{k}_{yu'}}+m_{\bm{k}_{v}}-n-1\} | i = 1, \ldots, n\}, \\
                & \min \{ m_{\bm{k}_{yu'}}, m_{\bm{k}_{v}}-n-1, m_{\bm{k}_{yu'}}+m_{\bm{k}_{v}}-n-1\}\} \\
                = & \min \{ m_{\bm{k}_{yu'}}-n, m_{\bm{k}_{v}}-n-1, m_{\bm{k}_{yu'}}+m_{\bm{k}_{v}}-n-1\} \\
                = & m_{\bm{k}_{yu'}}+m_{\bm{k}_{v}}-n-1 \\
                = & \min \{ m_{\bm{k}_{d^n yu'}}, m_{\bm{k}_{dv}}, m_{\bm{k}_{d^n yu'}}+m_{\bm{k}_{dv}}\}.
              \end{align*}
              In the penultimate, the equality follows from $m_{\bm{k}_{v}}-1, m_{\bm{k}_{yu'}} \leq 0$, and in the last, the equality follows from $m_{\bm{k}_{d^n yu'}}, m_{\bm{k}_{dv}} \leq 0$.
        \end{enumerate}
  \end{enumerate}
\end{proof}

\begin{cor}\label{cor ad shuffle}
  The space $\mathrm{span}_{\QQ}\{\Ical^{\mathrm{adm}}\}$ of admissible integer indices is a $\QQ$-subalgebra of $(\mathrm{span}_{\QQ}\{\Ical\}, \shuffle)$.
\end{cor}

\begin{proof}
  Let $\bm{k}, \bm{k}'$ be two admissible integer indices. By Proposition \ref{prop main shuffle}, we have
  $$m_{\bm{k} \shuffle \bm{k}'} = \min \{ m_{\bm{k}}, m_{\bm{k}'}, m_{\bm{k}}+m_{\bm{k}'}\} > 0.$$
  This shows that it is closed under the shuffle product, hence we have a subalgebra of $(\mathrm{span}_{\QQ}\{\Ical\}, \shuffle)$.
\end{proof}

\begin{cor}\label{cor re shuffle}
  The space $\mathrm{span}_{\QQ}\{\Ical^{\mathrm{reg}}\}$ of regularizable integer indices is a $\QQ$-subalgebra of $(\mathrm{span}_{\QQ}\{\Ical\}, \shuffle)$.
\end{cor}

\begin{proof}
  Let $\bm{k}, \bm{k}'$ be two regularizable integer indices. By Proposition \ref{prop main shuffle}, we have
  $$m_{\bm{k} \shuffle \bm{k}'} = \min \{ m_{\bm{k}}, m_{\bm{k}'}, m_{\bm{k}}+m_{\bm{k}'}\} \geq 0.$$
  This shows that it is closed under the shuffle product, hence we have a subalgebra of $(\mathrm{span}_{\QQ}\{\Ical\}, \shuffle)$.
\end{proof}

\begin{thm}\label{thm shuffle-pi+}
  Let $\mathrm{span}_{\QQ}\{\Ical^{\mathrm{adm}}\} = (\mathrm{span}_{\QQ}\{\Ical^{\mathrm{adm}}\}, \shuffle)$, $\mathrm{span}_{\QQ}\{\Ical_{>0}^{\mathrm{adm}}\} = (\mathrm{span}_{\QQ}\{\Ical_{>0}^{\mathrm{adm}}\}, \shuffle)$ be the $\QQ$-algebra with shuffle product $\shuffle$. Then, the positive-index map $\pi^{+}$ is a $\QQ$-algebra homomorphism, that is, the following diagram is commutative.
  $$\begin{tikzcd}
    \mathrm{span}_{\QQ}\{\Ical^{\mathrm{adm}}\} \times \mathrm{span}_{\QQ}\{\Ical^{\mathrm{adm}}\} \arrow[r, "\pi^{+} \otimes \pi^{+}"] \arrow[d, "\shuffle"] & \mathrm{span}_{\QQ}\{\Ical_{>0}^{\mathrm{adm}}\} \times \mathrm{span}_{\QQ}\{\Ical_{>0}^{\mathrm{adm}}\} \arrow[d, "\shuffle"'] \\
    \mathrm{span}_{\QQ}\{\Ical^{\mathrm{adm}}\} \arrow[r, "\pi^{+}"]                & \mathrm{span}_{\QQ}\{\Ical_{>0}^{\mathrm{adm}}\}                
  \end{tikzcd}$$
\end{thm}

\begin{proof}
  We claim the general formula below
  $$\pi^{+} \circ \shuffle \circ (\pi^{+} \otimes \pi^{+}) =  \pi^{+} \circ \shuffle : \mathrm{span}_{\QQ}\{\Ical\} \times \mathrm{span}_{\QQ}\{\Ical\} \to \mathrm{span}_{\QQ}\{\Ical\},$$
  whose restriction to $\mathrm{span}_{\QQ}\{\Ical^{\mathrm{adm}}\} \times \mathrm{span}_{\QQ}\{\Ical^{\mathrm{adm}}\}$ yields our claim. Note that the image of $\shuffle \circ (\pi^{+} \otimes \pi^{+})$ is in $\mathrm{span}_{\QQ}\{\Ical_{>0}^{\mathrm{adm}}\}$, and the linear map $\pi^{+}$ on $\mathrm{span}_{\QQ}\{ \Ical_{>0}^{\mathrm{adm}}\}$ is the identity map. Hence, the equation
  $$\shuffle \circ (\pi^{+} \otimes \pi^{+}) = \mathrm{id} \circ \shuffle \circ (\pi^{+} \otimes \pi^{+}) =  \pi^{+} \circ \shuffle$$
  holds, when restricted to $\mathrm{span}_{\QQ}\{\Ical^{\mathrm{adm}}\} \times \mathrm{span}_{\QQ}\{\Ical^{\mathrm{adm}}\}$.\\
  Let $\bm{k}, \bm{k}'$ be two admissible integer indices. We prove the claim by induction on the sum of the lengths of $w_{\bm{k}}$ and $w_{\bm{k}'}$. For $|w_{\bm{k}}| + |w_{\bm{k}'}| = 0$, it is trivial. For $|w_{\bm{k}}| + |w_{\bm{k}'}| = 1$, it suffices to prove the case for $y \shuffle \bm{1}$, that is,
  $$\pi^{+} \circ \shuffle \circ (\pi^{+} \otimes \pi^{+})((0), \emptyset) = \pi^{+} \circ \shuffle (\pi^{+}(0) \otimes \pi^{+}(\emptyset)) = \pi^{+} ((0) \shuffle \emptyset),$$
  and it follows that $\pi^{+}(0) = (0)$ and $\pi^{+}(\emptyset) = \emptyset$. For $|w_{\bm{k}}| + |w_{\bm{k}'}| = 2$,  it suffices to prove the four cases for
  $$jy \shuffle \bm{1}, \ dy \shuffle \bm{1}, \ yy \shuffle \bm{1}, \ y \shuffle y.$$
  It is easy to check
  $$\pi^{+} \circ \shuffle \circ (\pi^{+} \otimes \pi^{+})((1), \emptyset) = \pi^{+} \circ \shuffle (\pi^{+}(1) \otimes \pi^{+}(\emptyset)) = \pi^{+} ((1) \shuffle \emptyset),$$
  $$\pi^{+} \circ \shuffle \circ (\pi^{+} \otimes \pi^{+})((-1), \emptyset) = \pi^{+} \circ \shuffle (\pi^{+}(-1) \otimes \pi^{+}(\emptyset)) = \pi^{+} ((-1) \shuffle \emptyset),$$
  $$\pi^{+} \circ \shuffle \circ (\pi^{+} \otimes \pi^{+})((0, 0), \emptyset) = \pi^{+} \circ \shuffle (\pi^{+}(0, 0) \otimes \pi^{+}(\emptyset)) = \pi^{+} (\pi^{+}(0, 0)) = \pi^{+}(0, 0) = \pi^{+} ((0, 0) \shuffle \emptyset),$$
  $$\pi^{+} \circ \shuffle \circ (\pi^{+} \otimes \pi^{+})((0), (0)) = \pi^{+} \circ \shuffle (\pi^{+}(0) \otimes \pi^{+}(0)) = \pi^{+} ((0) \shuffle (0)).$$
  Assume $\pi^{+} \circ \shuffle \circ (\pi^{+} \otimes \pi^{+}) =  \pi^{+} \circ \shuffle$ holds for two integer indices $\bm{k}, \bm{k}'$ which satisfy $|w_{\bm{k}}| + |w_{\bm{k}'}| < n$. We prove the equation for $|w_{\bm{k}}| + |w_{\bm{k}'}| = n$ using Definition \ref{defn non-positive shuffle} and Remark \ref{rem non-positive shuffle}.
  \begin{enumerate}[\underline{\text{Case}}]
    \item 1. (i) in Definition \ref{defn non-positive shuffle}, i.e., $w_{\bm{k}} = yu$ and $w_{\bm{k}'} = v$. By Definition \ref{defn pi+} and Remark \ref{rem pi+}, it is easy to see $\pi^{+}(y\pi^{+}(u)) = \pi^{+}(yu)$. In this case, we obtain
        \begin{align*}
          \pi^{+}(yu \shuffle v) = & \pi^{+}(y(u \shuffle v)) \\
          = & \pi^{+}(y \pi^{+}(u \shuffle v)) \\
          = & \pi^{+}(y(\pi^{+}(u) \shuffle \pi^{+}(v)))
        \end{align*}
        by inductive hypothesis and
        $$\pi^{+}(\pi^{+}(yu) \shuffle \pi^{+}(v)) = \pi^{+}(\pi^{+}(y\pi^{+}(u)) \shuffle \pi^{+}(v)).$$
        It suffices to show that
        $$\pi^{+}(y(u' \shuffle v')) = \pi^{+}(\pi^{+}(yu') \shuffle v')$$
        for $u' \in \{w_{\bm{l}} | \bm{l} \in \mathrm{Supp}(\pi^{+}(\bm{k}))\}$ and $v' \in \{w_{\bm{l}'} | \bm{l}' \in \mathrm{Supp}(\pi^{+}(\bm{k}'))\}$.
        We consider the following cases.
        \begin{enumerate}[\underline{\text{Case}}]
          \item 1.1. $m_{\bm{k}_{u'}} \geq 0$. Since $u' \in \{w_{\bm{l}} | \bm{l} \in \mathrm{Supp}(\pi^{+}(\bm{k}))\}$, it is easy to see $\pi^{+}(yu') = yu'$ in this case. We use above to obtain
              \begin{align*}
                \pi^{+}(\pi^{+}(yu') \shuffle v') = & \pi^{+}(yu' \shuffle v') \\
                = & \pi^{+}(y(u' \shuffle v')).
              \end{align*}
          \item 1.2. $m_{\bm{k}_{u'}} < 0$. In this case, we prove the claim by induction on $m_{\bm{k}_{u'}}$. For $m_{\bm{k}_{u'}} = -1$, that is, $u' = yu''$. We obtain
              \begin{align*}
                \pi^{+}(\pi^{+}(yyu'') \shuffle v') = & \pi^{+}((dyu'' - ydu'' - yu'') \shuffle v') \\
                = & \pi^{+}(dy(u'' \shuffle v') - yd(u'' \shuffle v') - y(u'' \shuffle v'))
              \end{align*}
              and
              \begin{align*}
                \pi^{+}(y(yu'' \shuffle v')) = & \pi^{+}(yy(u'' \shuffle v')) \\
                = & \pi^{+}(dy(u'' \shuffle v') - yd(u'' \shuffle v') - y(u'' \shuffle v')).
              \end{align*}
              Assume $\pi^{+}(y(u' \shuffle v')) = \pi^{+}(\pi^{+}(yu') \shuffle v')$ holds for $u' \in \{w_{\bm{l}} | \bm{l} \in \mathrm{Supp}(\pi^{+}(\bm{k}))\}$ and $v' \in \{w_{\bm{l}'} | \bm{l}' \in \mathrm{Supp}(\pi^{+}(\bm{k}'))\}$ which satisfy $m_{\bm{k}_{u'}} > - m$. We prove the equation for $m_{\bm{k}_{u'}} = - m - 1$. It suffices to prove the case for $u' = d^{m}yu''$. We obtain
              \begin{align*}
                & \pi^{+}(\pi^{+}(yd^{m}yu'') \shuffle v')\\
                = & \pi^{+}\left(\pi^{+}\left(\frac{1}{m+1} \sum_{i=0}^{m}\binom{m+1}{i} B^-_i d^{m+1-i}yu'' - \delta_{m, 0} yu'' \right.\right. \\
                & \qquad \qquad \left.\left.- \frac{1}{m+1} \sum_{i=0}^{m}\binom{m+1}{i} B^+_i yd^{m+1-i}u'' \right) \shuffle v'\right)\\
                = & \pi^{+}\left(\frac{1}{m+1} \sum_{i=0}^{m}\binom{m+1}{i} B^-_i \pi^{+}(d^{m+1-i}yu'') \shuffle v' \right. \\
                & \qquad \left. - \frac{1}{m+1} \sum_{i=0}^{m}\binom{m+1}{i} B^+_i \pi^{+}(yd^{m+1-i}u'') \shuffle v' \right).
              \end{align*}
              By Definition \ref{defn pi+} and Remark \ref{rem pi+}, it is easy to see $\pi^{+}(dw) = d\pi^{+}(w)$. Since $u' \in \{w_{\bm{l}} | \bm{l} \in \mathrm{Supp}(\pi^{+}(\bm{k}))\}$ and $v' \in \{w_{\bm{l}'} | \bm{l}' \in \mathrm{Supp}(\pi^{+}(\bm{k}'))\}$, we have
              \begin{align*}
                & \pi^{+}(\pi^{+}(yd^{m}yu'') \shuffle v')\\
                = & \pi^{+}\left(\frac{1}{m+1} \sum_{i=0}^{m}\binom{m+1}{i} B^-_i d^{m+1-i}\pi^{+}(yu'') \shuffle \pi^{+}(v') \right. \\
                & \qquad \left. - \frac{1}{m+1} \sum_{i=0}^{m}\binom{m+1}{i} B^+_i \pi^{+}(yd^{m+1-i}u'') \shuffle \pi^{+}(v') \right) \\
                = & \pi^{+}\left(\frac{1}{m+1} \sum_{i=0}^{m}\binom{m+1}{i} B^-_i \pi^{+}(d^{m+1-i}yu'' \shuffle v') \right. \\
                & \qquad \left. - \frac{1}{m+1} \sum_{i=0}^{m}\binom{m+1}{i} B^+_i \pi^{+}(yd^{m+1-i}u'' \shuffle v') \right) \\
                = & \pi^{+}\left(\frac{1}{m+1} \sum_{i=0}^{m}\binom{m+1}{i} B^-_i d^{m+1-i}yu'' \shuffle v' \right. \\
                & \qquad \left. - \frac{1}{m+1} \sum_{i=0}^{m}\binom{m+1}{i} B^+_i yd^{m+1-i}u'' \shuffle v' \right)
              \end{align*}
              by inductive hypothesis. We use
              $$d^{n}u \shuffle v = \sum_{j = 0}^{n} (-1)^{j} \binom{n}{j} d^{n - j}(u \shuffle d^{i}v)$$
              to obtain
              \begin{align*}
                & \pi^{+}(\pi^{+}(yd^{m}yu'') \shuffle v')\\
                = & \pi^{+}\left( \frac{1}{m+1} \sum_{i=0}^{m}\binom{m+1}{i} B^-_i \left(\sum_{j = 0}^{m+1-i} (-1)^{j} \binom{m+1-i}{j} d^{m+1-i - j}y(u'' \shuffle d^{j}v')\right) \right. \\
                & \qquad \left. - \frac{1}{m+1} \sum_{i=0}^{m}\binom{m+1}{i} B^+_i \left(y\sum_{j = 0}^{m+1-i} (-1)^{j} \binom{m+1-i}{j} d^{m+1-i - j}(u'' \shuffle d^{j}v')\right)\right)\\
                = & \pi^{+}\left( \frac{1}{m+1} \sum_{i=0}^{m}\binom{m+1}{i} B^-_i \left(\sum_{j = 0}^{m-i} (-1)^{j} \binom{m+1-i}{j} d^{m+1-i - j}y(u'' \shuffle d^{j}v')\right) \right. \\
                & \qquad \left. - \frac{1}{m+1} \sum_{i=0}^{m}\binom{m+1}{i} B^+_i \left(y\sum_{j = 0}^{m-i} (-1)^{j} \binom{m+1-i}{j} d^{m+1-i - j}(u'' \shuffle d^{j}v')\right)\right)\\
                & \qquad - (-1)^{m} \pi^{+}(y(u'' \shuffle d^{m}v'))\\
                = & \pi^{+}\left( \sum_{i=0}^{m} \sum_{j = 0}^{m-i} (-1)^{j} \frac{m!}{i!j!} \left( B^-_i d^{m+1-i - j}y(u'' \shuffle d^{j}v') - B^+_i d^{m+1-i - j}(u'' \shuffle d^{j}v') \right) \right)\\
                & \qquad - (-1)^{m} \pi^{+}(y(u'' \shuffle d^{m}v')).
              \end{align*}
              On the other hand, we have
              \begin{align*}
                & \pi^{+}(y(d^{m}yu'' \shuffle v'))\\
                = & \pi^{+}\left(y\sum_{j = 0}^{m} (-1)^{j} \binom{m}{j} d^{m - j}y(u'' \shuffle d^{j}v')\right)\\
                = & \pi^{+}\left(\sum_{j = 0}^{m} (-1)^{j} \binom{m}{j} \pi^{+}(yd^{m - j}y(u'' \shuffle d^{j}v'))\right) \\
                = & \pi^{+}\left(\sum_{j = 0}^{m} (-1)^{j} \binom{m}{j} \pi^{+}\left(\frac{1}{m-j+1} \sum_{i=0}^{m-j}\binom{m-j+1}{i} B^-_i d^{m-j+1-i}y(u'' \shuffle d^{j}v')\right. \right. \\
                & \left. \left.  - \delta_{m-j, 0} y(u'' \shuffle d^{j}v') - \frac{1}{m-j+1} \sum_{i=0}^{m-j}\binom{m-j+1}{i} B^+_i yd^{m-j+1-i}(u'' \shuffle d^{j}v') \right)\right) \\
                = & \pi^{+}\left(\sum_{j = 0}^{m} (-1)^{j} \binom{m}{j} \pi^{+}\left(\frac{1}{m-j+1} \sum_{i=0}^{m-j}\binom{m-j+1}{i} B^-_i d^{m-j+1-i}y(u'' \shuffle d^{j}v')\right. \right. \\
                & \left. \left. - \frac{1}{m-j+1} \sum_{i=0}^{m-j}\binom{m-j+1}{i} B^+_i yd^{m-j+1-i}(u'' \shuffle d^{j}v') \right)\right) - (-1)^{m} \pi^{+}(y(u'' \shuffle d^{m}v')) \\
                = & \pi^{+}\left( \sum_{j=0}^{m} \sum_{i = 0}^{m-j} (-1)^{j} \frac{m!}{i!j!} \left( B^-_i d^{m+1-i - j}y(u'' \shuffle d^{j}v') - B^+_i d^{m+1-i - j}(u'' \shuffle d^{j}v') \right) \right) \\
                & - (-1)^{m} \pi^{+}(y(u'' \shuffle d^{m}v')) \\
                = & \pi^{+}\left( \sum_{i=0}^{m} \sum_{j = 0}^{m-i} (-1)^{j} \frac{m!}{i!j!} \left( B^-_i d^{m+1-i - j}y(u'' \shuffle d^{j}v') - B^+_i d^{m+1-i - j}(u'' \shuffle d^{j}v') \right) \right)\\
                & \qquad - (-1)^{m} \pi^{+}(y(u'' \shuffle d^{m}v')).
              \end{align*}
        \end{enumerate}
    \item 2. (ii) in Definition \ref{defn non-positive shuffle}, i.e., $w_{\bm{k}} = ju$ and $w_{\bm{k}'} = jv$. By Definition \ref{defn pi+} and Remark \ref{rem pi+}, it is easy to see $j\pi^{+}(u) = \pi^{+}(ju)$. In this case, we obtain
        \begin{align*}
          \pi^{+}(ju \shuffle jv) = & \pi^{+}(j(u \shuffle jv) + j(ju \shuffle v)) \\
          = & \pi^{+}(j(\pi^{+}(u \shuffle jv) + \pi^{+}(ju \shuffle v))) \\
          = & \pi^{+}(j(\pi^{+}(u) \shuffle \pi^{+}(jv) + \pi^{+}(ju) \shuffle \pi^{+}(v))) \\
          = & \pi^{+}(j(\pi^{+}(u) \shuffle j\pi^{+}(v) + j\pi^{+}(u) \shuffle \pi^{+}(v))) \\
          = & \pi^{+}((j\pi^{+}(u)) \shuffle (j\pi^{+}(v))) \\
          = & \pi^{+}(\pi^{+}(ju) \shuffle \pi^{+}(jv)).
        \end{align*}
        by inductive hypothesis.
    \item 3. (iii)–(v)  in Definition \ref{defn non-positive shuffle}, i.e., $w_{\bm{k}} = d^{m}yu$, where we used Remark \ref{rem non-positive shuffle}. By Definition \ref{defn pi+} and Remark \ref{rem pi+}, it is easy to see $d\pi^{+}(u) = \pi^{+}(du)$. We use
        $$d^{n}u \shuffle v = \sum_{j = 0}^{n} (-1)^{j} \binom{n}{j} d^{n - j}(u \shuffle d^{i}v)$$
        to obtain
        \begin{align*}
          \pi^{+}(d^{m}yu \shuffle v) = & \pi^{+}\left(\sum_{j = 0}^{n} (-1)^{j} \binom{n}{j} d^{n - j}y(u \shuffle d^{i}v)\right) \\
          = & \pi^{+}\left(\sum_{j = 0}^{n} (-1)^{j} \binom{n}{j} \pi^{+}(d^{n - j}y(u \shuffle d^{i}v))\right) \\
          = & \pi^{+}\left(\sum_{j = 0}^{n} (-1)^{j} \binom{n}{j} d^{n - j}\pi^{+}(y(u \shuffle d^{i}v))\right) \\
          = & \pi^{+}\left(\sum_{j = 0}^{n} (-1)^{j} \binom{n}{j} d^{n - j}\pi^{+}(yu) \shuffle \pi^{+}(d^{i}v)\right) \\
          = & \pi^{+}\left(\sum_{j = 0}^{n} (-1)^{j} \binom{n}{j} d^{n - j}\pi^{+}(yu) \shuffle d^{i}\pi^{+}(v)\right) \\
          = & \pi^{+}(d^{n}\pi^{+}(yu) \shuffle \pi^{+}(v))) \\
          = & \pi^{+}(\pi^{+}(d^{n}yu) \shuffle \pi^{+}(v))).
        \end{align*}
        by inductive hypothesis and Case 1.
  \end{enumerate}
\end{proof}

\begin{rem}
  By the same argument of the proof of Theorem \ref{thm shuffle-pi+}, we obtain the positive-index map $\pi^{+}$ is also a $\QQ$-algebra homomorphism between the algebra of regularizable integer index, that is, the following diagram is commutative.
  $$\begin{tikzcd}
    \mathrm{span}_{\QQ}\{\Ical^{\mathrm{reg}}\} \times \mathrm{span}_{\QQ}\{\Ical^{\mathrm{reg}}\} \arrow[r, "\pi^{+} \otimes \pi^{+}"] \arrow[d, "\shuffle"] & \mathrm{span}_{\QQ}\{\Ical_{>0}^{\mathrm{reg}}\} \times \mathrm{span}_{\QQ}\{\Ical_{>0}^{\mathrm{reg}}\} \arrow[d, "\shuffle"'] \\
    \mathrm{span}_{\QQ}\{\Ical^{\mathrm{reg}}\} \arrow[r, "\pi^{+}"]                & \mathrm{span}_{\QQ}\{\Ical_{>0}^{\mathrm{reg}}\}                
  \end{tikzcd}$$
\end{rem}

\section{Stuffle product for integer indices}\label{section Stuffle product for integer indices}
In this section, we recall the stuffle product of multiple polylogarithms with integer indices and prove that the stuffle product of two admissible (resp. regularizable) integer indices can be expanded as a linear combination of integer indices that remain admissible (resp. regularizable). We show that the positive-index map is an algebra homomorphism with respect to stuffle product (Theorem \ref{thm stuffle-pi+}).

We start by recalling recall the stuffle product for integer indices.
\begin{defn}\label{defn non-positive stuffle}
  We define the stuffle product
  $$\ast : \mathrm{span}_{\QQ}\{\Ical\} \times \mathrm{span}_{\QQ}\{\Ical\} \rightarrow \mathrm{span}_{\QQ}\{\Ical\}$$
  recursively by:
  $$\emptyset \ast \bm{k} = \bm{k} \ast \emptyset \coloneqq \bm{k}$$
  where $\bm{k} \in \Ical$ is an integer index, and for two integer indices $(\bm{k}, k), (\bm{k}', k')$ with $k, k'\in \ZZ$ and $\bm{k}, \bm{k}'\in \Ical$,
  $$(\bm{k}, k) \ast (\bm{k}', k') \coloneqq (\bm{k} \ast (\bm{k}', k'), k) + ((\bm{k}, k) \ast \bm{k}', k') + (\bm{k}\ast \bm{k}', k+k').$$
\end{defn}

The space $\mathrm{span}_{\QQ}\{\Ical\}$ is a $\QQ$-algebra with the stuffle product.

\begin{lem}\label{lem m_k stuffle}
  Let $\bm{k}$ be an integer index and $k$ be an integer. Then,
  $$m_{(\bm{k}, k)} = \min\{m_{\bm{k}}+k-1, k-1\}.$$
\end{lem}

\begin{proof}
  We rewrite $\bm{k}$ as $(k_1, \ldots, k_r)$. By Definition \ref{defn m}, we obtain
  $$m_{(\bm{k}, k)} = m_{(k_1, \ldots, k_r, k)} = m_{\bm{k}}+k-1$$
  if $m_{\bm{k}} \leq 0$;
  $$m_{(\bm{k}, k)} = m_{(k_1, \ldots, k_r, k)} = k-1$$
  if $m_{\bm{k}} > 0$.
\end{proof}

We will use Lemma \ref{lem m_k stuffle} several times in the proof of Proposition \ref{prop main stuffle}.

\begin{prop}\label{prop main stuffle}
  Let $\bm{k}, \bm{k}'$ be two integer indices. Then, 
  $$m_{\bm{k} \ast \bm{k}'} = \min \{ m_{\bm{k}}, m_{\bm{k}'}, m_{\bm{k}}+m_{\bm{k}'}\}.$$
\end{prop}

\begin{proof}
  We prove the claim by induction on $n = \mathrm{dep}(\bm{k}) + \mathrm{dep}(\bm{k}')$. For $n = 0$, it is trivial. For $n = 1$,  it suffices to prove the case for $(k) \ast \emptyset$, that is,
  $$m_{(k) \ast \emptyset} = \min \{ m_{(k)}, m_{\emptyset}, m_{(k)}+m_{\emptyset}\},$$
  and it follows that $m_{(k)} = k-1$ and $m_{\emptyset} = \infty$. For $n = 2$,  it suffices to prove the two cases for
  $$(k, k') \ast \emptyset, \ (k) \ast (k').$$
  It is easy to check
  $$m_{(k, k') \ast \emptyset} = \min \{k'-1, k+k'-2\} = \min \{ m_{(k, k')}, m_{\emptyset}, m_{(k, k')}+m_{\emptyset}\},$$
  $$m_{(k) \ast (k')} = \min \{k-1, k'-1, k+k'-2\} = \min \{ m_{(k)}, m_{(k')}, m_{(k)}+m_{(k')}\}.$$
  
  Assume $m_{\bm{k} \ast \bm{k}'} = \min \{ m_{\bm{k}}, m_{\bm{k}'}, m_{\bm{k}}+m_{\bm{k}'}\}$ holds for two integer indices $\bm{k}, \bm{k}'$ which satisfy $\mathrm{dep}(\bm{k}) + \mathrm{dep}(\bm{k}') < n$. We prove it for $\mathrm{dep}(\bm{k}) + \mathrm{dep}(\bm{k}') = n$. We write $\bm{k}$ and $\bm{k}'$ as $(k_1, \ldots, k_r)$ and $(k'_1, \ldots, k'_{r'})$ respectively. By Definition \ref{defn non-positive stuffle}, we get
  \begin{align*}
    \bm{k} \ast \bm{k}' = & ((k_1, \ldots, k_{r-1}) \ast (k'_1, \ldots, k'_{r'}), k_r) \\
    & + ((k_1, \ldots, k_r) \ast (k'_1, \ldots, k'_{r'-1}), k'_{r'}) \\
    & + ((k_1, \ldots, k_{r-1}) \ast (k'_1, \ldots, k'_{r'-1}), k_r + k'_{r'}).
  \end{align*}
  By inductive hypothesis,
  \begin{align*}
    m_{(k_1, \ldots, k_{r-1}) \ast (k'_1, \ldots, k'_{r'})} = & \min \{ m_{(k_1, \ldots, k_{r-1})}, m_{(k'_1, \ldots, k'_{r'})}, m_{(k_1, \ldots, k_{r-1})}+m_{(k'_1, \ldots, k'_{r'})}\}, \\
    m_{(k_1, \ldots, k_r) \ast (k'_1, \ldots, k'_{r'-1})} = & \min \{ m_{(k_1, \ldots, k_r)}, m_{(k'_1, \ldots, k'_{r'-1})}, m_{(k_1, \ldots, k_r)}+m_{(k'_1, \ldots, k'_{r'-1})}\}, \\
    m_{(k_1, \ldots, k_{r-1}) \ast (k'_1, \ldots, k'_{r'-1})} = & \min \{ m_{(k_1, \ldots, k_{r-1})}, m_{(k'_1, \ldots, k'_{r'-1})}, m_{(k_1, \ldots, k_{r-1})}+m_{(k'_1, \ldots, k'_{r'-1})}\}.
  \end{align*}
  These imply
  \begin{align*}
    m_{\bm{k} \ast \bm{k}'} = & \min \{m_{(k_1, \ldots, k_{r-1}) \ast (k'_1, \ldots, k'_{r'})} + k_r - 1, k_r - 1, \\
    & m_{(k_1, \ldots, k_r) \ast (k'_1, \ldots, k'_{r'-1})} + k'_{r'} - 1, k'_{r'} - 1, \\
    & m_{(k_1, \ldots, k_{r-1}) \ast (k'_1, \ldots, k'_{r'-1})} + k_r + k'_{r'} - 1, k_r + k'_{r'} - 1\} \\
    = & \min \{m_{(k_1, \ldots, k_{r-1})} + k_r - 1, m_{(k'_1, \ldots, k'_{r'})} + k_r - 1, m_{(k_1, \ldots, k_{r-1})}+m_{(k'_1, \ldots, k'_{r'})} + k_r - 1, k_r - 1, \\
    & m_{(k_1, \ldots, k_r)} + k'_{r'} - 1, m_{(k'_1, \ldots, k'_{r'-1})} + k'_{r'} - 1, m_{(k_1, \ldots, k_r)}+m_{(k'_1, \ldots, k'_{r'-1})} + k'_{r'} - 1, k'_{r'} - 1, \\
    & m_{(k_1, \ldots, k_{r-1})} + k_r + k'_{r'} - 1, m_{(k'_1, \ldots, k'_{r'-1})} + k_r + k'_{r'} - 1, \\
    & m_{(k_1, \ldots, k_{r-1})}+m_{(k'_1, \ldots, k'_{r'-1})} + k_r + k'_{r'} - 1, k_r + k'_{r'} - 1\} \\
    = & \min \{m_{(k_1, \ldots, k_r)}, m_{(k'_1, \ldots, k'_{r'}, k_r)}, m_{(k_1, \ldots, k_r)}+m_{(k'_1, \ldots, k'_{r'})}, \\
    & m_{(k_1, \ldots, k_r, k'_{r'})}, m_{(k'_1, \ldots, k'_{r'})}, m_{(k_1, \ldots, k_r)}+m_{(k'_1, \ldots, k'_{r'})}, \\
    & m_{(k_1, \ldots, k_r)} + k'_{r'}, m_{(k'_1, \ldots, k'_{r'})} + k_r, m_{(k_1, \ldots, k_{r})}+m_{(k'_1, \ldots, k'_{r'})}+1\} \\
    = & \min \{m_{(k_1, \ldots, k_r)}, m_{(k'_1, \ldots, k'_{r'}, k_r)}, m_{(k_1, \ldots, k_r)}+m_{(k'_1, \ldots, k'_{r'})}, m_{(k_1, \ldots, k_r, k'_{r'})}, m_{(k'_1, \ldots, k'_{r'})}\} \\
    = & \min \{m_{(k_1, \ldots, k_r)}, m_{(k'_1, \ldots, k'_{r'})}, m_{(k_1, \ldots, k_r)}+m_{(k'_1, \ldots, k'_{r'})}\}.
  \end{align*}
  In the third-to-last equality, we used several elementary properties of the minimum, including
  $$\min\{a+c,b+c\}=\min\{a,b\}+c,$$
  $$\min\{\min\{a,c\},\min\{b,c\}\}=\min\{a,b,c\},$$
  $$\min\{a,c\}+\min\{b,d\}\leq \min\{a+b,c+d\}.$$
  In the penultimate equality, we note that
  $$m_{(k_1, \ldots, k_r, k'_{r'})} \leq m_{(k_1, \ldots, k_r)} + k'_{r'},$$
  $$m_{(k'_1, \ldots, k'_{r'}, k_r)} \leq m_{(k'_1, \ldots, k'_{r'})} + k_r.$$
  In the final equality, we observed that
  $$\min \{m_{(k_1, \ldots, k_r)}, m_{(k_1, \ldots, k_r)}+m_{(k'_1, \ldots, k'_{r'})}\} \leq m_{(k_1, \ldots, k_r, k'_{r'})},$$
  $$\min \{m_{(k'_1, \ldots, k'_{r'})}, m_{(k_1, \ldots, k_r)}+m_{(k'_1, \ldots, k'_{r'})}\} \leq m_{(k'_1, \ldots, k'_{r'}, k_r)}.$$
\end{proof}

\begin{cor}\label{cor ad stuffle}
  The space $\mathrm{span}_{\QQ}\{\Ical^{\mathrm{adm}}\}$ of admissible integer indices is a $\QQ$-subalgebra of $(\mathrm{span}_{\QQ}\{\Ical\}, \ast)$.
\end{cor}

\begin{proof}
  Let $\bm{k}, \bm{k}'$ be two admissible integer indices. By Proposition \ref{prop main stuffle}, we have
  $$m_{\bm{k} \ast \bm{k}'} = \min \{ m_{\bm{k}}, m_{\bm{k}'}, m_{\bm{k}}+m_{\bm{k}'}\} > 0.$$
  This shows that it is closed under the stuffle product, hence we have a subalgebra of $(\mathrm{span}_{\QQ}\{\Ical\}, \ast)$.
\end{proof}

\begin{cor}\label{cor re stuffle}
  The space $\mathrm{span}_{\QQ}\{\Ical^{\mathrm{reg}}\}$ of regularizable integer indices is a $\QQ$-subalgebra of $(\mathrm{span}_{\QQ}\{\Ical\}, \ast)$.
\end{cor}

\begin{proof}
  Let $\bm{k}, \bm{k}'$ be two regularizable integer indices. By Proposition \ref{prop main stuffle}, we have
  $$m_{\bm{k} \ast \bm{k}'} = \min \{ m_{\bm{k}}, m_{\bm{k}'}, m_{\bm{k}}+m_{\bm{k}'}\} \geq 0.$$
  This shows that it is closed under the stuffle product, hence we have a subalgebra of $(\mathrm{span}_{\QQ}\{\Ical\}, \ast)$.
\end{proof}

\begin{thm}\label{thm stuffle-pi+}
  Let $\mathrm{span}_{\QQ}\{\Ical^{\mathrm{adm}}\} = (\mathrm{span}_{\QQ}\{\Ical^{\mathrm{adm}}\}, \ast)$, $\mathrm{span}_{\QQ}\{\Ical_{>0}^{\mathrm{adm}}\} = (\mathrm{span}_{\QQ}\{\Ical_{>0}^{\mathrm{adm}}\}, \ast)$ be the $\QQ$-algebra with stuffle product $\ast$. Then, the positive-index map $\pi^{+}$ is a $\QQ$-algebra homomorphism, that is, the following diagram is commutative.
  $$\begin{tikzcd}
    \mathrm{span}_{\QQ}\{\Ical^{\mathrm{adm}}\} \times \mathrm{span}_{\QQ}\{\Ical^{\mathrm{adm}}\} \arrow[r, "\pi^{+} \otimes \pi^{+}"] \arrow[d, "\ast"] & \mathrm{span}_{\QQ}\{\Ical_{>0}^{\mathrm{adm}}\} \times \mathrm{span}_{\QQ}\{\Ical_{>0}^{\mathrm{adm}}\} \arrow[d, "\ast"'] \\
    \mathrm{span}_{\QQ}\{\Ical^{\mathrm{adm}}\} \arrow[r, "\pi^{+}"]                & \mathrm{span}_{\QQ}\{\Ical_{>0}^{\mathrm{adm}}\}                
  \end{tikzcd}$$
\end{thm}

\begin{proof}
  Let $\bm{k}, \bm{k}'$ be two admissible integer indices. We prove the claim by induction on $n = \mathrm{dep}(\bm{k}) + \mathrm{dep}(\bm{k}')$. For $n = 0$, it is trivial. For $n = 1$,  it suffices to prove the case for $(k) \ast \emptyset$, that is,
  $$\pi^{+}(k) \ast \pi^{+}(\emptyset) = \pi^{+}(k) = \pi^{+}((k) \ast \emptyset)$$
  and it follows that $\pi^{+}(k) = (k)$ and $\pi^{+}(\emptyset) = \emptyset$. For $n = 2$,  it suffices to prove the two cases for
  $$(k, k') \ast \emptyset, \ (k) \ast (k').$$
  It is easy to check
  $$\pi^{+}(k, k') \ast \pi^{+}(\emptyset) = \pi^{+}(k, k') = \pi^{+}((k, k') \ast \emptyset),$$
  $$\pi^{+}(k) \ast \pi^{+}(k') = (k) \ast (k') = \pi^{+}((k) \ast (k')).$$
  Assume $\ast \circ (\pi^{+} \otimes \pi^{+}) =  \pi^{+} \circ \ast$ holds for two admissible integer indices $(\bm{k}, k), (\bm{k}', k')$ which satisfy $\mathrm{dep}(\bm{k}, k) + \mathrm{dep}(\bm{k}', k') < n$. We prove it for $\mathrm{dep}(\bm{k}, k) + \mathrm{dep}(\bm{k}', k') = n$ using Definition \ref{defn non-positive stuffle}. By Definition \ref{defn pi+} and Remark \ref{rem pi+}, it is easy to see $\pi^{+}(\bm{k}, k) = \pi^{+}(\pi^{+}(\bm{k}), k)$. By inductive hypothesis, we obtain
  \begin{align*}
    & \ast \circ (\pi^{+} \otimes \pi^{+})((\bm{k}, k) \ast (\bm{k}', k')) \\
    = & \pi^{+}(\bm{k}, k) \ast \pi^{+}(\bm{k}', k') \\
    = & \pi^{+}(\pi^{+}(\bm{k}, k) \ast \pi^{+}(\bm{k}', k')) \\
    = & \pi^{+}(\pi^{+}(\pi^{+}(\bm{k}), k) \ast \pi^{+}(\pi^{+}(\bm{k}'), k')) \\
    = & \pi^{+}((\pi^{+}(\pi^{+}(\bm{k}), k) \ast \pi^{+}(\bm{k}'), k') + (\pi^{+}(\bm{k}) \ast \pi^{+}(\pi^{+}(\bm{k}'), k'), k) + (\pi^{+}(\bm{k}) \ast \pi^{+}(\bm{k}'), k + k')) \\
    = & \pi^{+}((\pi^{+}(\bm{k}, k) \ast \pi^{+}(\bm{k}'), k') + (\pi^{+}(\bm{k}) \ast \pi^{+}(\bm{k}', k'), k) + (\pi^{+}(\bm{k}) \ast \pi^{+}(\bm{k}'), k + k')) \\
    = & \pi^{+}((\pi^{+}((\bm{k}, k) \ast (\bm{k}')), k') + (\pi^{+}(\bm{k} \ast (\bm{k}', k')), k) + (\pi^{+}(\bm{k} \ast \bm{k}'), k + k')) \\
    = & \pi^{+}(\pi^{+}((\bm{k}, k) \ast (\bm{k}')), k') + \pi^{+}(\pi^{+}(\bm{k} \ast (\bm{k}', k')), k) + \pi^{+}(\pi^{+}(\bm{k} \ast \bm{k}'), k + k') \\
    = & \pi^{+}((\bm{k}, k) \ast (\bm{k}'), k') + \pi^{+}(\bm{k} \ast (\bm{k}', k'), k) + \pi^{+}(\bm{k} \ast \bm{k}', k + k') \\
    = & \pi^{+}(((\bm{k}, k) \ast (\bm{k}'), k') + (\bm{k} \ast (\bm{k}', k'), k) + (\bm{k} \ast \bm{k}', k + k')) \\
    = & \pi^{+}((\bm{k}, k) \ast (\bm{k}', k')).
  \end{align*}
\end{proof}

\begin{rem}
  By the same argument of the proof of Theorem \ref{thm stuffle-pi+}, we obtain the positive-index map $\pi^{+}$ is also a $\QQ$-algebra homomorphism between the algebra of regularizable integer index, that is, the following diagram is commutative.
  $$\begin{tikzcd}
    \mathrm{span}_{\QQ}\{\Ical^{\mathrm{reg}}\} \times \mathrm{span}_{\QQ}\{\Ical^{\mathrm{reg}}\} \arrow[r, "\pi^{+} \otimes \pi^{+}"] \arrow[d, "\ast"] & \mathrm{span}_{\QQ}\{\Ical_{>0}^{\mathrm{reg}}\} \times \mathrm{span}_{\QQ}\{\Ical_{>0}^{\mathrm{reg}}\} \arrow[d, "\ast"'] \\
    \mathrm{span}_{\QQ}\{\Ical^{\mathrm{reg}}\} \arrow[r, "\pi^{+}"]                & \mathrm{span}_{\QQ}\{\Ical_{>0}^{\mathrm{reg}}\}                
  \end{tikzcd}$$
\end{rem}

\section{Applications}\label{section applications}
In this section, we recall the definition of $p$-adic multiple zeta values and extend it to the class of admissible integer indices (Definition \ref{defn p-adic MZV integer index}). We further show that both the shuffle and stuffle products extend naturally to this setting; hence the double shuffle relation holds for $p$-adic multiple zeta values of admissible integer indices (Theorem \ref{thm p-adic MZV integer index DSR}).

\begin{defn}[\cite{p-adicMZV}]\label{defn p-adic MPL}
  Fix a prime $p$. For an admissible positive integer index $\bm{k}=(k_1,\ldots,k_r) \in \ZZ_{>0}^r$, the $p${\it -adic (single variable) multiple polylogarithms} are defined by the following series
  $$\Li^p_{k_1,\ldots,k_r}(z) = \sum_{0<n_1<\cdots<n_r} \frac{z^{n_r}}{n_1^{k_1}\cdots n_r^{k_r}} \in \QQ_p\llbracket z \rrbracket.$$
\end{defn}

The following remark follows from Definition \ref{defn pi+}.

\begin{rem}\label{rem p-adic MPL integer indices}
  Let $\bm{k}=(k_1,\ldots,k_r) \in \ZZ^r$ be an admissible integer index. Let $\pi^{+}$ be the positive-index map in Definition \ref{defn pi+}. Then by equation \eqref{eq pi+},
  \begin{equation}\label{eq p-adic Li}
    \Li^p_{\bm{k}}(z) = \sum_{\bm{l} \in \mathrm{Supp}(\pi^{+}(\bm{k}))}c_{\bm{k}, \bm{l}}\Li^p_{\bm{l}}(z)
  \end{equation}
  holds in $\QQ_p\llbracket z \rrbracket$ for the coefficient $c_{\bm{k}, \bm{l}}\in \QQ \setminus \{0\}$ in Definition \ref{defn pi+}.
\end{rem}

The $p$-adic multiple polylogarithms $\Li^p_{\bm{k}}(z)$ converge on $\left\{ z \in \CC_p \middle| |z|_p < 1 \right\}$. By Coleman’s theory of $p$-adic iterated integration, one can regard it as a Coleman function
\begin{equation}\label{eq Coleman iterated integral}
  \Li^p_{k_1,\ldots,k_r}(z) = \int_{0}^{z}\omega_1 \omega_0^{k_r-1}\cdots \omega_1 \omega_0^{k_1-1}
\end{equation}
on $X \coloneqq \PP^1\setminus\{0,1,\infty\}$, where $\omega_1 = \frac{dt}{1-t}, \omega_0 = \frac{dt}{t}$, which allows us to locally analytically extends $\Li^p_{k_1,\ldots,k_r}(z)$ to $X(\CC_p)$ (cf. \cite[Definition~2.9]{p-adicMZV}). In a neighborhood of $z = 1 - t = 1$, $\Li^p_{k_1,\ldots,k_r}(1-t)$ admits an expansion as a $\QQ_p\llbracket t \rrbracket$-coefficient polynomial in $\log^{\alpha}(t)$ on $0<|t|_p<\varepsilon$, where $\alpha \in \CC_p$ is a branch parameter. A specific limit $\lim_{z\to 1}\!'\Li^p_{k_1,\ldots,k_r}(z)$ is concerned in \cite[Notation~2.12]{p-adicMZV}. It is shown that it converges for admissible integer indices and independent of any choice of branch parameter $\alpha$ by \cite[Theorem~2.13, Theorem~2.18]{p-adicMZV}. Furusho \cite{p-adicMZV} considers this limit to define $p$-adic multiple zeta values as follows:

\begin{defn}[{\cite[Definition~2.17, Theorem~2.25]{p-adicMZV}}]\label{defn p-adic MZV}
  Let $\bm{k}=(k_1,\ldots,k_r)\in \ZZ_{>0}^r$ be an admissible positive integer index. The $p${\it -adic multiple zeta values} are defined by the following specific limit
  $$\zeta_p(k_1,\ldots,k_r) \coloneqq \lim_{z\to 1}\!'\Li^p_{k_1,\ldots,k_r}(z) \in \QQ_p.$$
\end{defn}

By using Remark \ref{rem p-adic MPL integer indices} and Definition \ref{defn p-adic MZV}, we define the $p$-adic multiple zeta values of integer indices as follows:

\begin{defn}\label{defn p-adic MZV integer index}
  Let $\bm{k}=(k_1,\ldots,k_r)\in \ZZ^r$ be an admissible integer index. We define the $p${\it -adic multiple zeta value} $\zeta_p(k_1,\ldots,k_r)$ for an integer index by
  $$\zeta_p(k_1,\ldots,k_r) \coloneqq \lim_{z\to 1}\!'\Li^p_{\bm{k}}(z) = \sum_{\bm{l} \in \mathrm{Supp}(\pi^{+}(\bm{k}))}c_{\bm{k}, \bm{l}}\lim_{z\to 1}\!'\Li^p_{\bm{l}}(z) = \sum_{\bm{l} \in \mathrm{Supp}(\pi^{+}(\bm{k}))}c_{\bm{k}, \bm{l}}\zeta_p(\bm{l}) \in \QQ_p,$$
  where the second equal is due to equation \eqref{eq p-adic Li}.
\end{defn}

We consider the $\QQ$-linear map
$$\zeta_p : \mathrm{span}_{\QQ}\{\Ical^{\mathrm{adm}}\} \to \QQ_p$$
defined by $\bm{k} \mapsto \zeta_p(\bm{k})$.

\begin{rem}\label{rem p-adic double shuffle admissible positive}
  The $p$-adic multiple zeta values for admissible positive integer indices satisfy the double shuffle relation, that is, the $\QQ$-linear map $\zeta_p$ is the $\QQ$-algebra homomorphism with respect to the shuffle product $\shuffle$
  \begin{equation}\label{eq p-adic shuffle admissible positive}
    \zeta_p : (\mathrm{span}_{\QQ}\{\Ical_{>0}^{\mathrm{adm}}\}, \shuffle) \to \QQ_p
  \end{equation}
  and is the $\QQ$-algebra homomorphism with respect to the stuffle product $\ast$
  \begin{equation}\label{eq p-adic stuffle admissible positive}
    \zeta_p : (\mathrm{span}_{\QQ}\{\Ical_{>0}^{\mathrm{adm}}\}, \ast) \to \QQ_p
  \end{equation}
  by \cite{p-adicDS, p-adicRDS}.
\end{rem}

\begin{prop}\label{prop p-adic MZV integer index shuffle}
  Let $(\mathrm{span}_{\QQ}\{\Ical^{\mathrm{adm}}\}, \shuffle)$ be the $\QQ$-algebra of admissible integer indices. Then, the linear map
  $$\zeta_p : \mathrm{span}_{\QQ}\{\Ical^{\mathrm{adm}}\} \to \QQ_p$$
  is a $\QQ$-algebra homomorphism, i.e., $\zeta_p(\bm{k} \shuffle \bm{k}') = \zeta_p(\bm{k})\zeta_p(\bm{k}')$ its restriction to admissible positive integer indices agree with \eqref{eq p-adic shuffle admissible positive}.
\end{prop}

\begin{proof}
  By Theorem \ref{thm shuffle-pi+} and Remark \ref{rem p-adic double shuffle admissible positive}, the composition of algebra homomorphism is still an algebra homomorphism, which proves the proposition.
\end{proof}

\begin{prop}\label{prop p-adic MZV integer index stuffle}
  Let $(\mathrm{span}_{\QQ}\{\Ical^{\mathrm{adm}}\}, \ast)$ be the $\QQ$-algebra of admissible integer indices. Then, the linear map
  $$\zeta_p : \mathrm{span}_{\QQ}\{\Ical^{\mathrm{adm}}\} \to \QQ_p$$
  is a $\QQ$-algebra homomorphism, i.e., $\zeta_p(\bm{k} \ast \bm{k}') = \zeta_p(\bm{k})\zeta_p(\bm{k}')$ its restriction to admissible positive integer indices agree with \eqref{eq p-adic stuffle admissible positive}.
\end{prop}

\begin{proof}
  By Theorem \ref{thm stuffle-pi+} and Remark \ref{rem p-adic double shuffle admissible positive}, the composition of algebra homomorphism is still an algebra homomorphism, which proves the proposition.
\end{proof}

\begin{thm}\label{thm p-adic MZV integer index DSR}
  Let $\bm{k}, \bm{k}'$ be admissible integer indices. The $p$-adic multiple zeta values at integer indices satisfy the double shuffle relation, that is,
  $$\zeta_p(\bm{k} \shuffle \bm{k}') = \zeta_p(\bm{k})\zeta_p(\bm{k}') = \zeta_p(\bm{k} \ast \bm{k}').$$
\end{thm}

\begin{proof}
  By Proposition \ref{prop p-adic MZV integer index shuffle} and Proposition \ref{prop p-adic MZV integer index stuffle}, one can see that $\zeta_p(\bm{k})\zeta_p(\bm{k}') = \zeta_p(\bm{k} \shuffle \bm{k}')$ and $\zeta_p(\bm{k})\zeta_p(\bm{k}') = \zeta_p(\bm{k} \ast \bm{k}')$, which proves the theorem.
\end{proof}

\begin{eg}
  For admissible integer indices $(a)$ and $(b, c)$ with $b < 0$, we have
  \begin{align*}
    & \zeta_p((a) \shuffle (b, c)) \\
    = & \sum_{i = 0}^{a} \binom{c+i}{i} \zeta_p((b)\shuffle (a-i),c+i) + \sum_{i = 0}^{c} \binom{a+i}{i} \zeta_p(b,c-i,a+i) \\
    = & \sum_{i = 0}^{a} \binom{c+i}{i} \left[ \sum_{j = 0}^{\min\{a-i-1, -b\}} (-1)^{j} \binom{-b}{j} \zeta_p(a-i-j,b+j,c+i) \right. \\
    & \left. + (-1)^{a-i} \sum_{j=0}^{-b-a+i} \binom{-b-1-j}{a-i-1} \zeta_p(-j,b+a-i+j,c+i) \right]\\
    & + \sum_{i = 0}^{c} \binom{a+i}{i} \zeta_p(b,c-i,a+i)
  \end{align*}
  and
  \begin{align*}
    & \zeta_p((a) \ast (b, c)) \\
    = & \zeta_p(b, c, a) + \zeta_p((a) \ast (b), c) + \zeta_p(b, a+c) \\
    = & \zeta_p(b, c, a) + \zeta_p(b, a, c) + \zeta_p(a, b, c) + \zeta_p(a+b, c) + \zeta_p(b, a+c),
  \end{align*}
  where we formally let $\binom{n}{-1} = \delta_{n,-1}$.
\end{eg}

\begin{rem}
  The above arguments also work in the complex case: the series \eqref{eq MZV} converges for any admissible integer index $\bm{k} \in \Ical^{\mathrm{adm}}$. Its value $\zeta(\bm{k})$ coincides with $\zeta \circ \pi^{+}(\bm{k})$ and satisfies the double shuffle relation as given in Theorem \ref{thm p-adic MZV integer index DSR}.
\end{rem}

\textbf{Acknowledgments.} The author would like to thank Professor Hidekazu Furusho for his patient supervision, and for suggesting the topic of this work.


\begin{thebibliography}{9}
\bibitem{q-MZVDSR}
H. Bachmann.
\newblock {\em Double shuffle relations for $q$-analogues of multiple zeta values, their derivatives and the connection to multiple Eisenstein series}.
\newblock RIMS Kôkyûroku. no. 2015, pp. 22-43. (2016)

\bibitem{p-adicDS}
A. Besser and H. Furusho.
\newblock {\em The double shuffle relations for $p$-adic multiple zeta values}.
\newblock Contemp. Math. Volume 416, pp. 9-29. (2006)

\bibitem{Euler's formula}
D. M. Bradley.
\newblock {\em A $q$-analog of Euler's decomposition formula for the double zeta function}.
\newblock Int. J. Math. Math. Sci. no. 21, pp. 3453-3458. (2005)

\bibitem{Broadhurst}
D. J. Broadhurst and D. Kreimer.
\newblock {\em Association of multiple zeta values with positive knots via Feynman diagrams up to $9$ loops}.
\newblock Phys. Lett. B. Volume 393, no. 3-4, pp. 403-412. (1997)

\bibitem{non-positive shuffle}
K. Ebrahimi-Fard, D. Manchon and J. Singer.
\newblock {\em The Hopf algebra of ($q$-)multiple polylogarithms with non-positive arguments}.
\newblock Int. Math. Res. Not. IMRN. no. 16, pp. 4882-4922. (2017)

\bibitem{p-adicMZV}
H. Furusho.
\newblock {\em $p$-adic multiple zeta values I: $p$-adic multiple polylogarithms and the $p$-adic KZ equation}.
\newblock Invent. Math. Volume 155, no. 2, pp. 253-286. (2004)

\bibitem{p-adicMZV2}
H. Furusho.
\newblock {\em $p$-adic multiple zeta values II: Tannakian interpretations}.
\newblock  Amer. J. Math. Volume 129, no. 4, pp. 1105-1144. (2007)

\bibitem{p-adicRDS}
H. Furusho and A. Jafari.
\newblock {\em Regularization and generalized double shuffle relations for $p$-adic multiple zeta values}.
\newblock Compos. Math. Volume 143, no. 5, pp. 1089-1107. (2007)

\bibitem{Goncharov}
A. B. Goncharov and Y. I. Manin.
\newblock {\em Multiple $\zeta$-motives and moduli spaces $\overline{\Mcal}_{0,n}$}.
\newblock Compos. Math. Volume 140, no. 1, pp. 1-14. (2004)

\bibitem{EMZVDSR}
P. Lochak, N. Matthes and L. Schneps.
\newblock {\em Elliptic multizetas and the elliptic double shuffle relations}.
\newblock Int. Math. Res. Not. IMRN. no. 1, pp. 698-756. (2021)

\bibitem{Matsumoto}
K. Matsumoto.
\newblock {\em The analytic continuation and the asymptotic behaviour of certain multiple zeta-functions. I}.
\newblock J. Number Theory. Volume 101, no. 2, pp. 223-243. (2003)

\end{thebibliography}
\end{document}